\documentclass[smallextended,referee,envcountsect,]{svjour3}
\smartqed

\usepackage{amsmath, amssymb, amsthm}
\usepackage[a4paper, margin=2.5cm]{geometry}  
\usepackage{graphicx}
\usepackage{comment}
\usepackage{placeins}
\usepackage{enumerate}
\usepackage{hyperref}

\usepackage[short]{optidef}	

\usepackage{xcolor}

\newcommand{\R}{\mathbb{R}}
\newcommand{\A}{\mathcal{A}}

\newcommand{\dd}{\mathrm{d}}

\newcommand{\I}{{\mathcal{I}}}
\newcommand{\J}{{\mathcal{J}}}

\newcommand{\B}{{\mathcal{B}}}

\newcommand{\BNLP}{{\mathrm{BNLP}}}

\definecolor{darkgreen}{rgb}{0.0,0.5,0.0}

\journalname{-}

\begin{document}

\title{Local Convergence Results for Sequential Quadratic Programming with Complementarity Constraints}

\author{Armin Nurkanovi\'c}

\institute{A. Nurkanovi\'c, Corresponding author  \at
             University of Freiburg\\
              Freiburg, 79110 Freiburg, Germany\\
              armin.nurkanovic@imtek.uni-freiburg.de
}

\date{Received: 04.2026 / Accepted: date}

\maketitle

\begin{abstract}
Mathematical programs with complementarity constraints (MPCCs) are a challenging class of nonlinear optimization problems, because their nonlinear programming reformulations violate standard constraint qualifications at every feasible point.
This paper analyzes sequential quadratic programming with complementarity constraints (SQPCC). 
In this method, the complementarity constraints are retained in the subproblems, yielding quadratic programs with complementarity constraints (QPCCs). 
Therefore, SQPCC is a reasonable extension of classical sequential quadratic programming (SQP) to MPCCs. 
The main contribution of the paper is a new local convergence result for the SQPCC method to S-stationary points. 
We show that there exists at least one sequence of QPCC S-stationary points converging to a reference S-stationary point of the MPCC, and we further characterize conditions under which each such sequence converges and under which such a sequence is locally unique.
In contrast to previous results, the analysis is established under weaker second-order sufficient conditions and requires neither classical strict complementarity nor upper-level strict complementarity. 
Our result builds upon local convergence results for the classical SQP method. 
In addition, we present improved local convergence results for classical SQP within the kappa-omega setting, which also cover linearly convergent variants. 
The kappa-omega assumptions, which are widely used in the analysis of Newton-type methods, provide a natural framework for quantifying subproblem approximation errors, their effect on the convergence speed, and the effect of nonlinearity on the size of the local convergence region. 
Furthermore, we establish an active-set stabilization result for SQPCC, identifying conditions under which the optimal complementarity and inequality active sets are identified after finitely many iterations, and conditions under which such identification occurs only asymptotically. 
Numerical examples illustrate the theoretical findings and highlight some advantages of SQPCC over classical SQP applied to MPCCs.
\end{abstract}


\section{Introduction}\label{sec:intro}
In this paper, we study the local convergence properties of a method for computing a stationary point of a mathematical program with complementarity constraints (MPCC):
\begin{mini!}[2]
	{\substack{w\in \R^{n}}}{f(w)\label{eq:mpcc_intro_obj}}
	{\label{eq:mpcc_intro}}{}
	\addConstraint{h(w)}{=0 \label{eq:mpcc_intro_eq}}
	\addConstraint{g(w)}{\leq 0 \label{eq:mpcc_intro_ineq}}
	\addConstraint{0 \leq G(w) \perp H(w)}{\geq 0, \label{eq:mpcc_intro_comp}}
\end{mini!}
where $f:\R^n \to \R$ is the objective function, $h:\R^n \to \R^{m_h}$ defines the equality constraints, and $g:\R^n \to \R^{m_g}$ defines the inequality constraints.
The functions $G,H:\R^n \to \R^m$ define the complementarity constraints. 
All problem functions are assumed to be twice continuously differentiable.
The notation in Eq. \eqref{eq:mpcc_intro_comp} means that $G(w), H(w) \geq 0$ and that $G_i(w)H_i(w) = 0$ for $i = 1,\ldots,m$.
By replacing~\eqref{eq:mpcc_intro_comp} with these conditions, the MPCC~\eqref{eq:mpcc_intro} becomes a standard nonlinear program (NLP). 
However, this reformulation, as well as other equivalent NLP reformulations; cf.~Sec.~\ref{sec:mpcc_theory}, violates standard constraint qualifications such as the Mangasarian--Fromovitz constraint qualification (MFCQ) at every feasible point~\cite{Jane2005,Scheel2000}. 
This obstructs the direct application of classical NLP theory and algorithms~\cite{Luo1996,Scheel2000}. 
Nevertheless, MPCCs admit a rich tailored theory and a broad range of numerical methods that address the limitations of standard NLP approaches; see, for example,~\cite{Kim2020,Luo1996,Scheel2000}.

Complementarity constraints are an expressive modeling tool, and MPCCs arise in many applications, including, but not limited to, bilevel optimization~\cite{Kim2020}, process engineering~\cite{Biegler2010}, robotics with contact and friction~\cite{Posa2014,Nurkanovic2024}, optimal control of hybrid dynamical systems~\cite{Nurkanovic2023f}, and optimal power flow~\cite{Rosehart2005}. 
Closely related problem classes include mathematical programs with vanishing constraints~\cite{Achtziger2008} and sparsity optimization problems~\cite{Feng2018,Kanzow2024}, both of which can be reformulated into an MPCC.

Most MPCC-tailored algorithms either regularize the complementarity constraints or exploit their piecewise structure, and then solve a (finite) sequence of standard NLPs converging to a stationary point of the original MPCC. 
Easy-to-implement representatives include penalty, smoothing, and regularization methods~\cite{Scholtes2001,Steffensen2010,Kanzow2013,Ralph2004,Raghunathan2005,Kadrani2009,Pozharskiy2026}; surveys comparing their numerical and theoretical properties can be found in~\cite{Hoheisel2013,Kanzow2015,Nurkanovic2024b}.
Active-set methods attempt to identify feasible or optimal active sets of the complementarity constraints, thereby reducing the MPCC locally to a standard NLP on which the iterations are performed. 
Classical approaches of this type are reported in~\cite{Giallombardo2008,Izmailov2008,Lin2006}, while more recent variants~\cite{Guo2022,Kirches2022,Nurkanovic2025} solve more elaborate subproblems and guarantee convergence to B-stationary points under weaker assumptions.

Most of the methods above require the repeated solution of NLPs associated with the MPCC to convergence, or at least to high accuracy. 
There is also a class of sequential quadratic programming (SQP) methods tailored to MPCCs~\cite{Chen2007,Fletcher2006,Giallombardo2008,Luo1996}.

In~\cite{Luo1996}, the authors propose a piecewise SQP method and establish its local convergence. 
The idea is to apply the classical SQP method to the ``pieces'' of the MPCC, namely to so-called branch NLPs; see also Sec.~\ref{sec:mpcc_theory}. 
These branch NLPs are obtained by fixing each complementarity pair either to $G_i(w)\geq 0,\ H_i(w)=0$ or to $G_i(w)=0,\ H_i(w)\geq 0$. 
Based on the values of the complementarity functions at the current iterate, a partition is selected and SQP is applied to the corresponding branch. 
Consequently, once the correct branch has been identified, the method inherits the local convergence properties of the classical SQP method~\cite[Theorem~6.4.3]{Luo1996}. 
Its main limitation is that this favorable behavior relies on the iterate being sufficiently close to a solution so that the optimal active set of nondegenerate complementarity pairs, that is, those satisfying $G_i(\bar w)+H_i(\bar w)>0$, has already been identified. 
Globalization strategies for such methods are discussed in~\cite{Giallombardo2008,Scholtes1999}.

Fletcher et al.~\cite{Fletcher2006} study the application of the classical SQP method directly to NLP reformulations of MPCCs. 
Their analysis assumes standard regularity conditions together with the availability of a quadratic programming (QP) solver capable of detecting a linearly independent active-set basis. 
This requirement is arguably restrictive in practice, since many modern QP solvers that perform well in applications~\cite{Kouzoupis2018} do not satisfy it. 
As in piecewise SQP, the theory also depends on assumptions ensuring that the relevant complementarity structure is identified in early iterates.

In this paper, we consider sequential quadratic programming with complementarity constraints (SQPCC), introduced by Scholtes~\cite{Scholtes2004}. SQPCC is arguably the most direct extension of SQP to MPCCs, since it retains the complementarity structure in the subproblems, which are then quadratic programs with complementarity constraints (QPCCs). 
A detailed local convergence theory was not the focus of~\cite{Scholtes2004}. 
There, only a sketch of a local convergence proof was provided, relying on strong second-order sufficient conditions, strict complementarity, and strict upper-level complementarity~\cite[Sec.~5]{Scholtes2004}; cf.~Sec.~\ref{sec:mpcc_theory}. 
By contrast, we show that SQPCC admits local convergence rates analogous to classical SQP for all locally convergent sequences of QPCC S-stationary points.
In particular, no explicit active-set stabilization or branch-selection assumptions are needed, since this role is handled implicitly by the QPCC subproblems. 
Moreover, the local convergence region can be strictly larger; cf.~Sec.~\ref{sec:sqpcc}.

It is clear that solving QPCC subproblems can be more difficult than solving QPs~\cite{Bai2013}. 
However, substantial recent progress in tailored QPCC solvers~\cite{Aydinoglu2024,DeMarchi2023,DeMarchi2025,Hall2024,Jia2023,Kirches2010y,Pozharskiy2026} has led to orders-of-magnitude speedups over naive approaches, making SQPCC increasingly practical from a computational perspective.

Our interest in the local convergence theory of SQPCC is also motivated by dynamic optimization, in particular by model predictive control. 
In standard parametric NLPs arising in dynamic optimization, the QP subproblems are closely related to implicit-function-theorem-based sensitivity results~\cite{Fiacco1983} for parametric NLPs~\cite{Diehl2001}. 
This link is lost for parametric MPCCs, and neither classical SQP nor piecewise SQP retains it. 
However, it was shown recently in~\cite{Nurkanovic2026a} that an analogous relation is preserved by QPCCs, which provides a basis for model predictive control of dynamic complementarity systems~\cite{Nurkanovic2023f}. 

The goal of the present paper is to provide a detailed understanding of the local convergence of the SQPCC method, rather than an extensive numerical benchmarking study. However, nontrivial optimal control examples solved with SQPCC are available in the GitHub repository of \texttt{nosnoc}~\cite{Nurkanovic2022b}.

\subsection{Contributions.}
The main contributions of this paper are as follows:
\begin{enumerate}
	\item In Theorem~\ref{th:sqp_local_convergence}, we reformulate a local convergence result for the classical SQP method for NLPs from~\cite{Bonnans1994} so that it also covers linearly convergent variants, such as Gauss--Newton Hessian approximations. 
	In comparison to~\cite{Bonnans1994}, our analysis is formulated in terms of the $\kappa$- and $\omega$-assumptions~\cite{Bock1987}, which are widely used in the local convergence analysis of Newton-type and SQP-type methods~\cite{Deuflhard2011,TranDinh2012b,Verschueren2016,Zanelli2021c}. 
	No strict complementarity or active-set stabilization assumptions are required. 
	The proof uses only elementary tools from nonlinear programming theory.
	
	\item In Section~\ref{sec:sqpcc} and Theorem~\ref{th:sqpcc_local_convergence}, we establish a new local convergence result to S-stationary points for the SQPCC method. 
	Our assumptions are MPCC counterparts of the regularity conditions used in the classical SQP theory, in particular MPCC--SSOSC and MPCC--LICQ (or the weaker MPCC-SMFCQ); see~Sec.~\ref{sec:mpcc_theory} for definitions.
	In contrast to previous results~\cite{Scholtes2004}, our analysis does not require any strict complementarity and allows for complementarity active-set changes during the iterations. 
	Moreover, we prove that stationary points weaker than S-stationary points can, if sufficiently close, attract iterates to S-stationary points, thereby enlarging the local convergence region.
	
	\item In addition, we clarify how multiple local sequences of QPCC S-stationary points may coexist, when spurious attracting sequences can occur, and under which conditions such sequences are excluded (cf. Corollary~\ref{cor:no_spurious_sequences}). 
	Moreover, we discuss conditions ensuring local uniqueness of the SQPCC iterates.
	
	\item In Proposition~\ref{th:active_set_stabilization_sqpcc}, we establish an active-set stabilization result for SQPCC. 
	More precisely, we identify conditions on the MPCC--Lagrange multipliers under which the optimal complementarity and inequality active sets are identified after finitely many iterations, as well as conditions under which this identification occurs only asymptotically.
	
	\item Numerical examples illustrate qualitative advantages of SQPCC over classical SQP applied to MPCC reformulations. 
	In particular, we show that Theorem~\ref{th:sqpcc_local_convergence} is not affected by a counterexample due to Leyffer and Munson~\cite{Leyffer2007}. 
	Moreover, we show that SQPCC can converge to an S-stationary point even in cases where SQP converges to a non-S-stationary point. 
	
\end{enumerate}

\subsection{Outline.}
Section~\ref{sec:nlp_theory} reviews basic concepts from nonlinear programming, and Section~\ref{sec:sqp} presents SQP methods together with a local convergence result and an active-set stabilization result. 
The MPCC counterparts of NLP regularity notions are provided in Section~\ref{sec:mpcc_theory}. 
Section~\ref{sec:sqpcc} derives the SQPCC method, provides a local convergence proof, an active-set stabilization result, and illustrative examples.
Section~\ref{sec:conclusion} discusses directions for future research.

\subsection{Notation.} An open ball centered at $\bar w \in \R^n$ with radius $\varepsilon >0$ is denoted by $\B_{\varepsilon}(\bar w) := \{ w \in \R^n \mid \| w- \bar w \| < \varepsilon \}$.
The concatenation of two vectors $ x \in \R^{n_x}, y \in \R^{n_y}$ is compactly denoted by $ (x,y) := [x^\top, y^\top]^\top$. 
The concatenation of several vectors is defined accordingly.
Subscripts of the form $w^k$ denote iterates, whereas subscripts of the form $w_j$ denote the $j$th component of a vector.
For the Jacobian of a function $f: \R^n \to \R^m$ we will use the notation $\nabla f(w) = \frac{\partial f}{\partial w}(w)^\top \in \R^{n \times m}$.
We use the symbol $L$ for the usual NLP Lagrangian, and calligraphic $\mathcal{L}$ for the MPCC--Lagrangian.

\section{Preliminaries on nonlinear programming}\label{sec:nlp_theory}
We state basic optimality and sensitivity results for standard nonlinear programs (NLPs). 
This is necessary to present the local convergence results for SQP and their MPCC counterparts in a self-contained manner.
\subsection{Regularity and local optimality}
Consider the following NLP:
\begin{mini!}[2]
	{\substack{w \in \R^{n}}}{f(w)\label{eq:nlp_obj}}
	{\label{eq:nlp}}{}
	\addConstraint{h(w)}{=0 \label{eq:nlp_eq}}
	\addConstraint{g(w)}{\leq0, \label{eq:nlp_ineq}}
\end{mini!}
where $f : \R^n \to \R$ and $h : \R^n \to \R^{m_h}$, $g : \R^n \to \R^{m_g}$ are twice continuously differentiable functions.
The Lagrange multipliers for the equality and inequality constraints are denoted by $\lambda \in \R^{m_h}$ and $\mu \in \R^{m_g}$, respectively.
The feasible set of the NLP~\eqref{eq:nlp} is denoted by $\Omega := \{ w\in \R^n \mid h(w) = 0, g(w) \leq 0\}$.
The Lagrange function $L: \R^n \times \R^{m_h} \times \R^{m_g} \to \R$ of the NLP~\eqref{eq:nlp} is defined as:
\begin{align*}
	L(w,\lambda, \mu) &= f(w) + \lambda^\top h(w) + \mu^\top g(w). 
\end{align*}

The Karush–Kuhn–Tucker (KKT) conditions read as:
\begin{subequations}\label{eq:kkt}
	\begin{align}
		&\nabla f(w) + \nabla h(w) \lambda + \nabla g(w) \mu = 0, \\
		&h(w) = 0,\\
		&0 \leq -g(w) \perp \mu \geq 0.
	\end{align}
\end{subequations}
A point $\bar z:= (\bar w,\bar \lambda,\bar \mu)$ satisfying these conditions is called a KKT point, and $\bar w$ is called a stationary point.
If in addition a constraint qualification holds (defined below), the KKT conditions are first-order necessary conditions for optimality.
The active set at a feasible point $w \in \Omega$ is defined as:
\begin{align}\label{eq:active_sets}
	\A(w) = \{ i \in \{1,\ldots,m_g \} \mid g_i(w) = 0\}.
\end{align}
Its complement $\A^{\mathrm{C}}(w) = \{1,\ldots, m_g \} \setminus \A(w)$ is the set of inactive constraints.
The strictly and weakly active sets at a KKT point $(w,\lambda,\mu)$ are defined as:
\begin{align}\label{eq:active_sets_strict}
	&\A_+(w,\mu) = \{ i \in \A(w) \mid \mu_i > 0\},\quad
	\A_0(w,\mu) = \{ i \in \A(w) \mid \mu_i = 0\}.
\end{align}
We say that strict complementarity holds for a KKT point if 
\(
\mu_i> 0, \ \forall i \in \A(w).
\)
Next, we define some basic regularity concepts in nonlinear programming.

\begin{definition}[LICQ]
	At a feasible point $w \in \Omega$, the linear independence constraint qualification (LICQ) is said to hold if the set of vectors $\{ \nabla h_i(w) \mid i = 1,\ldots, m_h\} \cup \{ \nabla g_i(w) \mid i \in \A(w)\}$ are linearly independent.
\end{definition}
The uniqueness of the Lagrange multipliers is equivalent to the strict Mangasarian--Fromovitz constraint qualification (SMFCQ)~\cite{Izmailov2014,Wachsmuth2013}, which is implied by LICQ.
In this paper, we will use LICQ over SMFCQ, since it is stated only in terms of the problem functions and does not require the multipliers~\cite{Wachsmuth2013}.
However, in all assumptions below the LICQ may be replaced by the SMFCQ (both in the NLPs and the MPCC counterparts).

We use the notation $\nabla g_{\I}(w) \in \R^{n \times |\I|}$ for a matrix whose columns are the vectors $\nabla g_i(w),\ i \in \I \subseteq \{1,\ldots,m_g\}$.
Next, we define some second-order sufficient conditions (SOSC).
For this purpose we define the critical cone at KKT point $(w,\lambda,\mu)$:
\begin{align*}
	\mathcal{C}(w,\mu) = \{ d\in \R^n \mid \nabla h(w)^\top d = 0,\ \nabla g_{\A_+(w,\mu)}(w)^\top d = 0,\ 
	\nabla g_{\A_0(w,\mu)}(w)^\top d \leq 0
	\}.
\end{align*}
The tangent space of all equality and strictly active constraints is denoted by:
\begin{align}\label{eq:strong_critical_std}
	\mathcal{C}^\mathrm{S}(w,\mu) = \{ d\in \R^n \mid \nabla h(w)^\top d = 0,\ \nabla g_{\A_+(w,\mu)}(w)^\top d = 0 
	\}.
\end{align}
This cone is also called the strong critical cone.
If strict complementarity holds we have that $\mathcal{C}^{\mathrm{S}}(w,\mu) = \mathcal{C}(w,\mu)$.
Otherwise, we have that 	
$\mathcal{C}(w,\mu) = \{d \in \mathcal{C}^{\mathrm{S}}(w,\mu)  \mid \nabla g_{\A_0(w,\mu)}(w)^\top d \leq 0 \} \subset \mathcal{C}^\mathrm{S}(w,\mu)$.
We will also use the so-called strong second-order sufficient conditions (SSOSC).
\begin{definition}
	The SSOSC is said to hold at a stationary point $w \in \Omega$ if there exist corresponding multipliers $(\lambda,\mu)$ such that $(w,\lambda,\mu)$ is a KKT point, and if the following inequality holds:
	\begin{align}\label{eq:ssosc}
		d^\top \nabla_{ww}^2 L(w,\lambda,\mu) d > 0,\ \forall d \in \mathcal{C}^{\mathrm{S}}(w,\mu) \setminus\{0\}.
	\end{align}
\end{definition}
A less restrictive notion is the usual second-order sufficient conditions, which is obtained if we replace $\mathcal{C}^{\mathrm{S}}(w,\mu)$ by $\mathcal{C}(w,\mu)$ in \eqref{eq:ssosc}.
Under strict complementarity the SOSC and SSOSC coincide.

\subsection{Stability of solutions}
Now we modify the NLP~\eqref{eq:nlp} by allowing that all problem functions also depend on a parameter $p \in \R^{n_p}$.
In our context, such formulations are useful for the convergence analysis of the SQP method.
For a parameter $p \in \R^{n_p}$, we regard the following NLP:
\begin{mini!}[2]
	{\substack{w\in \R^{n}}}{f(w,p)\label{eq:pnlp_obj}}
	{\label{eq:pnlp}}{}
	\addConstraint{h(w,p)}{=0 \label{eq:pnlp_eq}}
	\addConstraint{g(w,p)}{\leq0,\label{eq:pnlp_ineq}}
\end{mini!}
with $f: \R^n \times \R^{n_p} \to \R$, and $h:\R^n \times \R^{n_p}  \to \R^{m_h},\ g:\R^{n} \times \R^{n_p}  \to \R^{m_g}$ are twice continuously differentiable functions.
All notions introduced above are defined with respect to $w$, and their definitions extend readily to the parametric setting.
Next, we restate a classical result on the stability and sensitivity of solutions of parametric NLPs~\cite{Fiacco1983}.

\begin{theorem}[Theorem~2.4.5, \cite{Fiacco1983}]\label{th:fiacco_nlp_sensitivity}
	If $f$, $g$, and $h$ are twice continuously differentiable in $(w,p)$ in a neighborhood of $(\bar w,\bar p)$, if the SSOSC hold for problem \eqref{eq:pnlp} at a KKT point $(\bar w,\bar{\lambda},\bar{\mu})$ corresponding to a fixed $\bar p$, and if the LICQ holds at $\bar{w}$, then:
	\begin{enumerate}[(a)]
		
		\item $\bar{w}$ is an isolated local optimal solution of~\eqref{eq:pnlp} for $p = \bar p$, with $(\bar \lambda,\bar \mu)$ being the unique Lagrange multiplier vector associated with $\bar{w}$;
		
		\item there exist a neighborhood of $\bar p$, and in this neighborhood there exists a unique continuous vector function
		\(
		z(p) = (w(p), \lambda(p),  \mu(p)) 
		\)
		satisfying the SSOSC for a local minimum of	problem \eqref{eq:pnlp} such that
		$z(\bar p) = (\bar w ,\bar \lambda ,\bar \mu )$, and hence $w(p)$ is a locally unique minimizer of \eqref{eq:pnlp} with the parameter $p$, with associated unique Lagrange multipliers $\lambda(p)$ and $\mu(p)$;
		
		
		\item the solution map $z(p)$ is Lipschitz in $p$, i.e., there exist constants $0<\gamma<\infty$ and $\delta>0$ such that for every $p$ with $\|p-\bar{p}\|<\delta$ it holds that:
		\begin{align}\label{eq:lipschitz_solution_fiacco}
			\|z(p)-\bar z\|\leq \gamma \|p-\bar{p}\|
		\end{align}
	\end{enumerate}
\end{theorem}

It is well known that, due to continuity of the solution map $z(p)$ of~\eqref{eq:pnlp}, strictly active constraints remain active and inactive constraints remain inactive for all $p \in \B_{\rho}(\bar p)$ for some $\rho > 0$.
The size of this neighborhood depends on the associated Lagrange multipliers and the slack of the inactive constraints.
We state this fact explicitly for later reference.
\begin{corollary}\label{lem:active_set_stabilization}
	Regard the parametric NLP~\eqref{eq:pnlp} and suppose the assumptions of Theorem~\ref{th:fiacco_nlp_sensitivity} hold true for~\eqref{eq:pnlp}, with $(\bar w, \bar \lambda, \bar \mu)$ being the solution at a fixed $\bar p$.
	Then there exist a constant $\rho>0$, such that for every $p \in \B_{\rho} (\bar p)$ it holds that:
	\begin{align}\label{eq:active_set_general}
		\mathcal{A}_+(\bar w,\bar\mu)=\mathcal{A}_+(w(p),\mu(p)),
		\qquad
		\mathcal{A}^\mathrm{C}(\bar w) = \mathcal{A}^\mathrm{C}(w(p)).
	\end{align}
\end{corollary}
\textit{Proof.}
Let us define the switching function
$
\varphi_i(p) := \mu_i(p) + g_i(w(p)).
$
It follows from Theorem \ref{th:fiacco_nlp_sensitivity}~(c) that there exists a $\delta > 0$ such that, for all $p\in \B_{\delta}(\bar p)$, the solution map $z(p)$ is Lipschitz continuous.
Consequently, both $g(w(p))$ and $\mu(p)$ are locally Lipschitz continuous functions of $p$, and hence so is $\varphi_i(p)$.

Due to continuity we have for all strictly active constraints it holds for all $p$ near $\bar p $ that $\varphi_i(p) > 0$.
We can compute explicitly the radius for which this holds true for each constraint.
In particular, from Lipschitz continuity of the switching function, we have for a Lipschitz constant $L_{\varphi_i}>0$:
$
\varphi_i(p) - \varphi_i(\bar p) \geq - L_{\varphi_i}\|p - \bar p\|.
$
The constraint $g_i(w(p)) = 0 $ stays strictly active if 
\begin{align}\label{eq:active_set_on_mu}
	\|p - \bar p \| < \frac{\varphi_i(\bar p)}{L_{\varphi_i}} = \frac{\bar \mu_i}{L_{\varphi_i}} := \rho^\mu_i.
\end{align} 
By repeating a similar argument for inactive constraints and $\A^{\mathrm{C}}(\bar{w}) = \{i \mid \varphi_i( \bar p) < 0 \}$, using Lipschitz continuity we have that it remains $g_i(w(p)) > 0$ for:
\begin{align}\label{eq:active_set_on_g}
	\|p - \bar p \| < \frac{-\varphi_i(\bar p)}{L_{\varphi_i}} = \frac{ - g_i (\bar w)}{L_{\varphi_i}} := \rho^g_i.
\end{align}
By picking $\rho = \min\Big( \delta,  \min_{i \in \A^+(\bar w, \bar \mu)} \rho_i^\mu, \min_{i \in \A^\mathrm{C}(\bar w)} \rho_i^g  \Big)$ we have that 
for $p \in \B_{\rho}(\bar p)$ equation~\eqref{eq:active_set_general} holds true.
\qed

\section{Sequential quadratic programming (SQP)}\label{sec:sqp}
In this section, we explicitly recall the derivation of the SQP method for NLPs, as it will be useful for explaining the derivation of the SQPCC methods for MPCCs in Section~\ref{sec:sqpcc}.
Moreover, we provide local convergence and active-set stabilization results.
These results are needed to state and prove analogous results for the SQPCC method.
\subsection{Method statement}
We consider a \textit{piecewise linearization} of the KKT system~\eqref{eq:kkt} at the point $(w^k,\lambda^k,\mu^k)$.
In particular, the bilinear term in the complementarity conditions is not linearized; instead, each function entering the complementarity condition is linearized separately.
Denoting the displacement by $\Delta w^k = w^{k+1} - w^k$, this yields:
\begin{subequations}\label{eq:qp_kkt}
	\begin{align}
		&\nabla f(w^k) + \nabla h(w^k)\lambda^{k+1} + \nabla g(w^k)\mu^{k+1}
		+ \nabla_{ww}^2 L(w^k,\lambda^k,\mu^k) \Delta w^k = 0,\\
		&h(w^k) + \nabla h(w^k)^\top \Delta w^k = 0,\\
		&0 \leq -g(w^k) - \nabla g(w^k)^\top \Delta w^k \perp \mu^{k+1} \geq 0.
	\end{align}
\end{subequations}
Clearly, for $\Delta w = 0$, we have $(w^{k+1},\lambda^{k+1},\mu^{k+1}) = (\bar{w},\bar{\lambda},\bar{\mu})$, and we recover the KKT conditions of the NLP at $(\bar{w},\bar{\lambda},\bar{\mu})$.

In general, the piecewise linearized KKT conditions~\eqref{eq:qp_kkt} at a linearization point $(w^k,\lambda^k,\mu^k)$ represent the KKT conditions of a quadratic program (QP).
Additionally, we may approximate the exact Hessian $\nabla_{ww}^2 L(w^k,\lambda^k,\mu^k)$ by a positive definite matrix $H^k$, which may depend on $w^k$ or on $(w^k,\lambda^k,\mu^k)$.
This yields the following QP:
\begin{mini!}[2]
	{\substack{\Delta w \in \R^{n}}}{\nabla f(w^k)^\top \Delta w + \frac{1}{2} \Delta w^\top H^k \Delta w \label{eq:qp_obj}}
	{\label{eq:qp}}{}
	\addConstraint{h(w^k) + \nabla h(w^k)^\top \Delta w}{=0 \label{eq:qp_eq}}
	\addConstraint{g(w^k) + \nabla g(w^k)^\top \Delta w}{\leq 0 \label{eq:qp_ineq}},
\end{mini!}
with Lagrange multipliers $\lambda^{k+1}$ and $\mu^{k+1}$ associated with the equality and inequality constraints, respectively.

In the full-step SQP method, a sequence of iterates $\{z^{k}\}$ is generated by solving the QP~\eqref{eq:qp} at each linearization point $(w^k,\lambda^k,\mu^k)$, yielding a step $\Delta w^k$, updating $w^{k+1} = w^k + \Delta w^k$, and setting the multipliers at iteration $k\!+\!1$ equal to the QP multipliers.
If the initial iterate $z^0$ is chosen sufficiently close to a solution $\bar{z}$, then under appropriate conditions the sequence $\{z^{k}\}$ converges (super)linearly or quadratically to $\bar{z}$, as shown in Theorem~\ref{th:sqp_local_convergence} below.

\begin{remark}\label{rem:adjoint_sqp}
	The constraint Jacobians in~\eqref{eq:qp}~may be approximated by $A^k \approx \nabla h(w^k)^\top$ and $B^k \approx \nabla g(w^k)^\top$, which in turn requires a modification of the objective gradient in~\eqref{eq:qp_obj}.
	This yields so-called adjoint-based SQP methods~\cite{Diehl2010a}.
	All results in this paper, for both the SQP and the SQPCC method, extend to this case with straightforward modifications.
	In particular, the constraint qualification should not only hold with the gradients of $\nabla h_i (w^k)$ and $\nabla g_i (w^k)$, but also for their approximations $A_i^k$ and $B_i^k$, respectively.
	To keep the notation light and the exposition clear, we restrict our exposition to the case of exact constraint Jacobians as in~\eqref{eq:qp}.
\end{remark}

\subsection{Local convergence}
The local convergence properties of the SQP method are well understood, and many variants of the classical results are available~\cite{Izmailov2014}. 
Among the sharpest of these is the result of Bonnans~\cite{Bonnans1994}. 
Here, we present a variant of this result under the commonly used ``kappa'' and ``omega'' assumptions from the convergence analysis of Newton-type methods~\cite{Bock2007,TranDinh2012b,Verschueren2016,Zanelli2021c}; see, for example,~\cite[Theorem~8.7]{Rawlings2017}, as well as \cite[Section~3.1]{Bock1987} and \cite[Section~4.3]{Deuflhard2011}. 
To the best of our knowledge, an elementary SQP convergence result in the $\kappa$--$\omega$ setting, but without strict complementarity, and hence without assuming immediate identification of the optimal active set $\A(\bar w)$ by the SQP iterates, has not been stated explicitly in the literature. 
We derive such a result here, both because these assumptions provide useful insight into the convergence mechanism and because the same framework will later be extended to the SQPCC method in Sec.~\ref{sec:sqpcc}. 
Our proof builds on the techniques introduced in~\cite{Bonnans1994}, where superlinear and quadratic convergence are established.

For notational convenience, we collect all defining functions of the KKT conditions~\eqref{eq:kkt} into a single mapping $\Phi: \R^{n_z} \to \R^{n_z}$ with $n_z = n + m_h + m_g$:
\begin{align}\label{eq:kkt_ge_definition}
	\Phi(z) := \Big(
	\nabla_{w} L(w,\lambda,\mu),
	h(w),
	- g(w)
	\Big).
\end{align}
This mapping is commonly used to express the KKT conditions as a generalized equation and to interpret the SQP method as a Newton--Josephy method~\cite{Izmailov2014}.
In this paper, however, we use only elementary results from nonlinear programming theory, and the generalized equation framework is not required.

It is not difficult to see that the constraint Jacobians, $\nabla h(w)$ and $\nabla g(w)$, and the Lagrangian Hessian, $\nabla_{ww}^2 L(w,\lambda,\mu)$, are collected in the Jacobian matrix $\nabla \Phi(z)^\top$.
If a Hessian approximation $H^k \approx \nabla_{ww}^2 L(w^k,\lambda^k,\mu^k)$ is used instead, the corresponding matrix is denoted by $\nabla \tilde{\Phi}(z^k)^\top \in \R^{n_z\times n_z}$.

The ``omega'' condition is simply a Lipschitz continuity assumption on $\nabla \Phi(z)^\top$, with constant $\omega > 0$; see~\eqref{eq:omega_cond}.
This assumption is commonly used to obtain quadratic convergence of Newton methods and related algorithms~\cite{Deuflhard2011}.
Relaxing it to mere continuity at the solution typically weakens the result to superlinear convergence~\cite{Izmailov2014}.
A smaller $\omega$ can be interpreted as indicating a ``less nonlinear'' problem and, as we discuss after the proof, leads to a larger convergence region.

We also impose a compatibility ``kappa'' condition that bounds the error in the derivative approximations; see~\eqref{eq:kappa_cond}.
In the context of constrained optimization, to the best of the author's knowledge, this assumption was first considered in~\cite{Bock1987}.
We discuss its effect on the size of the convergence region and on the convergence speed after the proof below.

A similar convergence result to ours was established in~\cite[Theorem~3.7]{TranDinh2012b}.
However, this result does not directly recover superlinear or quadratic convergence of the SQP method because it uses a more restrictive ``kappa'' assumption, comparing $\nabla \tilde \Phi(z^k)^\top$ with the value at the solution, that is, $\nabla \Phi(\bar z)^\top$, rather than, as is more common~\cite{Bock1987,Deuflhard2011}, with the value at each iterate.
Weakening this assumption requires somewhat different arguments in the convergence proof.
Moreover, \cite[Theorem~3.7]{TranDinh2012b} assumes strong regularity, whereas here we use only elementary tools, which also allow weaker assumptions; cf.~Remark~\ref{rem:sqp_ass}.

\begin{theorem}[SQP local convergence]\label{th:sqp_local_convergence}
	Let $f:\R^n\to\R$, $h:\R^n\to\R^{m_h}$, and $g:\R^n\to\R^{m_g}$	be twice differentiable in a neighborhood of a point $\bar{w}\in\R^n$, with their second derivatives being continuous at $\bar{w}$. 
	Let $\bar{w}$ be a local solution of problem~\eqref{eq:nlp}, satisfying the LICQ and the SSOSC for the associated unique Lagrange multiplier	$(\bar\lambda,\bar\mu)\in\R^{m_h}\times\R^{m_g}$.
	Further, suppose there exist constants $\omega > 0$ and $\bar{\kappa} < \frac{1}{3\gamma}$, where $\gamma > 0$ is the Lipschitz constant of the solution map of the parametric NLP~\eqref{eq:perturbed_nlp}, and a sequence $\{\kappa^k\}$ with $\kappa^k \in [0,\bar \kappa]$ such that, for all $z^{k}$ and $z$, the following inequalities hold:
	\begin{subequations}\label{eq:kappa_omega_sqp}
		\begin{align}
			&\| \nabla \Phi (z)^\top - \nabla \Phi (z^k)^\top \| \leq \omega  \| z  - z^k \|, \label{eq:omega_cond}\\
			&\| \nabla \Phi (z^k)^\top - \nabla \tilde \Phi (z^k)^\top \| \leq \kappa ^k. \label{eq:kappa_cond}
		\end{align}
	\end{subequations}
	Then there exists a constant $\varepsilon>0$ such that for any starting point $z^0 = (w^0,\lambda^0,\mu^0)\in \B_{\varepsilon}(\bar z)$ with $\bar z = (\bar{w},\bar\lambda,\bar\mu)$, there exists a sequence $\{(w^k,\lambda^k,\mu^k)\}\subset\R^n\times\R^{m_h}\times\R^{m_g}$ such that, for $k=0,1,\dots$, the point $\Delta w^{k}$ is a stationary point of the QP~\eqref{eq:qp} with a positive definite Hessian approximation $H^k$, and $(\lambda^{k+1},\mu^{k+1})$ is an associated Lagrange multiplier.
	Any such sequence converges to $(\bar{w},\bar\lambda,\bar\mu)$, and the rate of convergence is linear satisfying the contraction inequality: 
	\begin{align}\label{eq:sqp_contraction_estimate}
		\| z^{k+1} - \bar z \| \leq \alpha^k \| z^{k} - \bar z \| + \beta  \| z^{k+1} - \bar z \|^2.
	\end{align}
	with constants $\alpha^k \in [0,1)$ for all $k$ and $\beta > 0$.
	If, in addition, $\{ \kappa^k \} \to 0$ then $\alpha^k \to 0$ and the convergence rate is superlinear, and if $\bar \kappa = 0$, then $\alpha^k = 0$ and the convergence rate is quadratic.
\end{theorem}
\textit{Proof.}
Regard the nonlinear optimization problem
\begin{mini!}[2]
	{\substack{w\in \R^{n}}}{f(w)\label{eq:perturbed_nlp_obj} - r_f^\top w}
	{\label{eq:perturbed_nlp}}{}
	\addConstraint{h(w)-r_g}{=0 \label{eq:perturbed_nlp_eq}}
	\addConstraint{g(w)+r_h}{\leq 0 \label{eq:perturbed_nlp_ineq}},
\end{mini!}
with parameter $r = (r_f, r_g, r_h) \in \R^{n} \times \R^{m_h} \times \R^{m_g}$.
Define the parameter $r^k$ as
\begin{align}\label{eq:pertubation_definition}
	r^k := \Phi(z^{k+1}) - \Phi(z^{k}) - \nabla \tilde{\Phi}(z^{k})^\top (z^{k+1}-z^{k}),
\end{align}
where $\Phi$ is defined in Eq.~\eqref{eq:kkt_ge_definition}.
It can be verified by direct computation that the KKT conditions of the QP~\eqref{eq:qp} coincide with the KKT conditions of the NLP~\eqref{eq:perturbed_nlp} with $r = r^k$, and hence both problems have the same solution $z^{k+1}$.
Moreover, for $z^{k+1} = z^k = \bar z$, one has $\bar r = 0$, and~\eqref{eq:perturbed_nlp} reduces to the NLP~\eqref{eq:nlp} with  the solution $\bar z$.

Since for $\bar r = 0$ the NLP~\eqref{eq:perturbed_nlp} satisfies LICQ and SSOSC at $\bar z$, Theorem~\ref{th:fiacco_nlp_sensitivity}~(c) can be applied.  
Therefore, there exist constants $\delta_1 > 0$ and $0 < \gamma < \infty$ such that, for any solution $z(r)$ of~\eqref{eq:perturbed_nlp} near $\bar z = z(0)$ with $\|r\| \leq \delta_1$, it holds that
\begin{align}\label{eq:lipschitz_perturbed_nlp}
	\| z(r) - \bar z \| \leq \gamma \| r \|.
\end{align}

Next, observe that the QP~\eqref{eq:qp} is an optimization problem parametric in $z^k$.  
By introducing the variable change $\Delta w = w - w^k$, an equivalent QP is obtained, and for $z^k = \bar{z}$ the solution of this QP coincides with the solution of the NLP~\eqref{eq:nlp}.  
In this notation, the QP solution mapping satisfies $z(\bar{z}) = \bar z$ and $z(z^k) = z^{k+1}$.  
Since $H^k$ is positive definite for all $k$, SSOSC and LICQ hold at the QP solution $z(\bar z)$ for the parameter value $\bar z$.  
Hence, the assumptions of Theorem~\ref{th:fiacco_nlp_sensitivity}~(c) are satisfied for the QP~\eqref{eq:qp}.  
Therefore, there exist constants $\delta_2 > 0$ and $0 < \tilde \gamma < \infty$ such that for all $z^k$ with $\| z^k - \bar z \| \leq \delta_2$, there exists a solution $z^{k+1} = z(z^k)$ of the QP that satisfies, by Theorem~\ref{th:fiacco_nlp_sensitivity}~(c):
\begin{align}\label{eq:qp_lipschitz}
	\| z^{k+1} - \bar{z} \| \le \tilde \gamma \| z^k - \bar z \|.
\end{align}
Moreover, under these conditions and by convexity of the QP for each $k$, Theorem~\ref{th:fiacco_nlp_sensitivity}~(b) implies that the solution $z^{k+1} = z(z^k)$ is unique for all $\| z^{k} - \bar z \| \leq \delta_2$.

Next, by using the triangle inequality and~\eqref{eq:qp_lipschitz}, a bound on the step size is obtained:
\begin{align}\label{eq:step_size_bound}
	\| z^{k+1} - z^k \| \le \| z^{k+1} - \bar z \| + \| \bar z - z^k \|
	\le (1+\tilde \gamma)\,\delta_2.
\end{align}

We proceed by bounding $\| r^k \|$. 
Add and subtract $\nabla \Phi(z^{k})^\top(z^{k+1}-z^{k})$ on the right-hand side of~\eqref{eq:pertubation_definition}, regroup terms, and apply the triangle inequality:
\begin{align*}
	\| r^k \|
	&\le \| (\nabla \Phi(z^{k})^\top - \nabla \tilde \Phi(z^{k})^\top)(z^{k+1}-z^{k}) \| \\
	&\quad + \| \Phi(z^{k+1}) - \Phi(z^{k}) - \nabla \Phi(z^{k})^\top(z^{k+1}-z^{k}) \|.
\end{align*}
To bound the second term, apply the mean value theorem together with the $\omega$-condition~\eqref{eq:omega_cond}:
\begin{align*}
	&\| \Phi(z^{k+1}) - \Phi(z^{k}) - \nabla \Phi(z^{k})^\top(z^{k+1} - z^{k}) \| \\
	&= \Big\| \int_{0}^{1}
	\big(\nabla \Phi(z^{k} + t(z^{k+1} - z^{k}))^\top - \nabla \Phi(z^{k})^\top \big) (z^{k+1} - z^{k}) \, \dd t \Big\| \\
	&\le \int_{0}^{1}
	\omega \, \| t (z^{k+1} - z^{k}) \| \, \dd t \cdot \|z^{k+1} - z^{k} \|
	\le \frac{\omega}{2} \|z^{k+1} - z^{k} \|^2.
\end{align*}
Additionally, using~\eqref{eq:kappa_cond} for the first term yields
\begin{align}\label{eq:kappa_omega_rk_bound}
	\| r^k \|
	\le \Big( \kappa^k + \frac{\omega}{2} \|z^{k+1} - z^{k} \| \Big)
	\|z^{k+1} - z^{k} \|.
\end{align}
Applying the step-size bound~\eqref{eq:step_size_bound} gives
\begin{align}\label{eq:bound_rk_via_delta2}
	\| r^k \|
	\le \Big( \kappa^k + \frac{\omega}{2}(1+\tilde \gamma)\delta_2 \Big)
	\|z^{k+1} - z^{k} \|.
\end{align}
We have by assumption that $\kappa^k < \bar \kappa < \frac{1}{3\gamma}$ and, if necessary, we shrink $\delta_2$ so that
\begin{align}\label{eq:convergence_radius_discussion}
	\delta_2 \leq \frac{2}{\omega(1+\tilde{\gamma})}
	\left(\frac{1}{3\gamma}-\kappa^k\right).
\end{align}
Then, from~\eqref{eq:bound_rk_via_delta2} we obtain 
\begin{align}\label{eq:r_k_gamma_bound}
	\| r^k \| \le \frac{1}{3 \gamma } \|z^{k+1} - z^{k} \|.
\end{align}

In order to combine \eqref{eq:lipschitz_perturbed_nlp} with \eqref{eq:r_k_gamma_bound}, it must hold that $\| r^k \| \leq \delta_1$.  
This is ensured for sufficiently small $\delta_2$, since
\(
\| r^k \| \leq \frac{1}{3\gamma}\|z^{k+1}-z^{k}\|
\leq \frac{(1+\tilde \gamma)\delta_2}{3\gamma} \leq \delta_1.
\)
Therefore, it holds that 
\begin{align*}
	\| z^{k+1} - \bar z\| \leq \frac{1}{3} \| z^{k+1} - z^{k}\| \leq \frac{1}{3}\| z^{k+1} - \bar z\|	+ \frac{1}{3}\| z^{k} - \bar z\|,
\end{align*}
which implies the contraction estimate
\begin{align}\label{eq:sqp_contraction}
	\| z^{k+1} - \bar z\| \le \frac{1}{2} \| z^{k} - \bar z\|.
\end{align}
By choosing $\| z^0 - \bar z\| \le \varepsilon \le \delta_2$ and applying \eqref{eq:sqp_contraction} inductively, the sequence $\{z^{k}\}$ converges linearly to $\bar z$.  
It remains to prove the claimed order of convergence.

We combine \eqref{eq:lipschitz_perturbed_nlp} with \eqref{eq:kappa_omega_rk_bound} and use the triangle and Cauchy–Schwarz inequalities to obtain
\begin{align*}
	\| z^{k+1} - \bar z\|
	\leq 
	\gamma \kappa^k \| z^{k+1} - \bar z\|
	+ \gamma \kappa^k \| z^{k} - \bar z\|
	+ \gamma \omega \| z^{k+1} - \bar z\|^2
	+ \gamma \omega \| z^{k} - \bar z\|^2.
\end{align*}
Since~\eqref{eq:sqp_contraction} holds, there exists $\eta > 0$ with $\|z^{k+1} - \bar z \| < \eta$, which can be made sufficiently small, such that
$\gamma \omega \| z^{k+1} - \bar z\|^2 \le \gamma \omega \eta \| z^{k+1} - \bar z\|$.
Collecting the terms in $\| z^{k+1} - \bar z\|$ and moving them to the left-hand side yields
\begin{align*}
	\| z^{k+1} - \bar z\|
	\le \frac{\gamma \kappa^k}{1 - \gamma \kappa^k - \gamma \omega \eta}
	\| z^{k} - \bar z\|
	+ \frac{\gamma \omega}{1 - \gamma \kappa^k - \gamma \omega \eta}
	\| z^{k} - \bar z\|^2.
\end{align*}
Define $\alpha^k := \frac{\gamma \kappa^k}{1 - \gamma \kappa^k - \gamma \omega \eta}$ and
$\beta := \frac{\gamma \omega}{1 - \gamma \bar \kappa - \gamma \omega \eta}$ to obtain~\eqref{eq:sqp_contraction_estimate}.
If $\bar \kappa = 0$, then $\alpha^k = 0$ for all $k$, and quadratic convergence follows.
If $\kappa^k \to 0$, then $\alpha^k \to 0$ and  $\| z^{k+1} - \bar z\| \leq \underbrace{(\alpha^k + \beta \| z^{k} - \bar z\|)}_{\to 0} \| z^{k} - \bar z\|$, which implies superlinear convergence.

\qed 

We interpret the $\kappa$--$\omega$ assumptions.
First, for the exact-Hessian SQP method, we have $\bar \kappa = 0$ (and thus $\alpha^k =0$), which recovers a local quadratic convergence rate.
For quasi-Newton methods, where $H^k \to \nabla_{ww}^2 L(w^k,\lambda^k,\mu^k)$, the error in~\eqref{eq:kappa_cond} shrinks due to $\kappa^k \to 0$, and the usual superlinear convergence is recovered.
If $\kappa^k \not\to 0$, for example when using a constant, or a Gauss--Newton Hessian approximation, only linear convergence is retained.

Second, the Lipschitz constant $\gamma$ of $z(r)$ in~\eqref{eq:lipschitz_perturbed_nlp} characterizes how sensitive the solution is to perturbations.
The less sensitive it is, the larger $\kappa^k < \frac{1}{3\gamma}$ may be, that is, the larger the admissible error in the Hessian approximation.
Moreover, a smaller derivative approximation error leads to a smaller $\kappa^k$, and consequently to a smaller $\alpha^k < 1$, so that the resulting linear convergence rate is faster.

Finally, we comment on the influence of our assumptions on the convergence radius, which is, by~\eqref{eq:convergence_radius_discussion}, upper bounded by $ \varepsilon < \frac{2}{\omega(1+\tilde{\gamma})} \left(\frac{1}{3\gamma}-\kappa^k\right)$.
It follows that the smaller the error in the QP data, that is, the smaller $\kappa^k$, the larger the admissible local convergence radius.
Likewise, the less the Jacobian of $\Phi(z)$ changes, that is, the smaller $\omega$, the larger this radius becomes.
An analogous interpretation applies to the constants $\gamma$ and $\tilde \gamma$.
In the limiting case of an exact Hessian and $\omega = 0$ (that is, constant constraint Jacobians and Lagrangian Hessian), the NLP reduces to a QP, and the local convergence radius is unbounded above.
In other words, starting from any initial point, the SQP method applied to a QP converges in one iteration.

\begin{remark}\label{rem:sqp_ass}
	For ease of exposition, we assumed SSOSC and LICQ in order to apply Theorem~\ref{th:fiacco_nlp_sensitivity} directly and obtain the estimates \eqref{eq:lipschitz_perturbed_nlp} and \eqref{eq:qp_lipschitz} in the proof below.
	Alternatively, semistability and hemistability arguments as in~\cite{Bonnans1994} can be used; these yield the same estimates under the weaker assumptions of SOSC and SMFCQ; cf.~\cite[Proposition~1.37]{Izmailov2014} and \cite[Proposition~3.37]{Izmailov2014}.
	In particular, the main change in the proof is that the estimates \eqref{eq:lipschitz_perturbed_nlp} and \eqref{eq:qp_lipschitz} would follow from semistability (\cite[Proposition~1.37]{Izmailov2014}) rather than from Theorem~\ref{th:fiacco_nlp_sensitivity}~(c).
	Thus, the assumptions stated here can be relaxed to SOSC and SMFCQ, recovering the sharpest known local convergence results for the SQP method~\cite{Bonnans1994}.
	In addition to the results in~\cite{Bonnans1994}, the present result also covers the linearly convergent constrained Gauss--Newton method, which is often used in state estimation and model predictive control applications~\cite{Rawlings2017}.
\end{remark}

\subsection{Active-set stabilization for the SQP method}

Notably, Theorem~\ref{th:sqp_local_convergence} does not require that the solutions of the QP~\eqref{eq:qp} and the NLP~\eqref{eq:nlp} at $\bar z$ share the same active set at every iterate $k$. 
In fact, the active sets may coincide only in the limit and may differ at all preceding iterates~\cite{Bonnans1994}. 
Using Corollary~\ref{lem:active_set_stabilization}, we recall a known active-set stabilization property of the SQP method~\cite{Bonnans1994,Izmailov2014}. 
We provide an explicit simple proof and derive explicit stabilization radii, which will be useful in several subsequent arguments for the SQPCC case in Sec.~\ref{sec:sqpcc}. 
The proof shows that the farther the solution is from an active-set change, as quantified by large margins $-g_i(\bar{w})>0$ for inactive constraints and large multipliers $\bar{\mu}_i>0$ for active constraints, the larger the radius $\rho^{\mathrm{s}}>0$ on which the active set remains stable.

\begin{proposition}[SQP active-set stabilization]\label{th:active_set_stabilization}
	Suppose that the assumptions of Theorem~\ref{th:sqp_local_convergence} hold.
	Let $(w^0,\lambda^0,\mu^0)\in \mathcal{B}_{\varepsilon}(\bar z)$.
	
	Then there exists a constant $\rho^{\mathrm{s}}>0$ such that, after finitely many iterations, the iterates satisfy
	$
	(w^k,\lambda^k,\mu^k)\in \mathcal{B}_{\rho^{\mathrm{s}}}(\bar z) \cap \mathcal{B}_{\varepsilon}(\bar z),
	$
	and the active set stabilizes in the sense that
	\begin{align}\label{eq:active_set_stabilization}
		\mathcal{A}_+(\bar w,\bar\mu)=\mathcal{A}_+(w^k,\mu^k),
		\qquad
		\A_+(\bar w,\bar \mu ) \subseteq  \mathcal{A}(w^k)\subseteq
		\mathcal{A}_+(\bar w,\bar\mu)\cup\mathcal{A}_0(\bar w,\bar\mu).
	\end{align}
	
	If, in addition, strict complementarity holds at $\bar w$, then
	$
	\mathcal{A}(w^k)=\mathcal{A}(\bar w).
	$
\end{proposition}
\textit{Proof.}
We make use of Corollary~\ref{lem:active_set_stabilization} and apply it to the perturbed NLP~\eqref{eq:perturbed_nlp}, parameterized by $p = r^k$.
Note that $r^k \to 0$ as $k \to \infty$, since the right-hand side of~\eqref{eq:kappa_omega_rk_bound} is uniformly bounded and $\|z^{k+1}-z^{k}\|\to 0$.
By adapting~\eqref{eq:active_set_on_mu} to this setting we obtain that a strictly active constraint $g_i(w^k) = 0$ with $\bar \mu_i >0$ stays such for:
\begin{align}\label{eq:active_set_on_r}
	\| r\| < \frac{\bar \mu_i}{L_{\varphi_i}}.
\end{align}
Next, we want to express this in the distance of $z^{k}$ to $\bar{z}$ for some finite $k$.
We know from~\eqref{eq:step_size_bound} in the proof of Theorem~\ref{th:sqp_local_convergence} that 
$\|z^{k+1} - z^{k} \| \leq (1+\tilde \gamma )\varepsilon^k$, where $\varepsilon^k:=\| z^{k} - \bar z \|$ and $\varepsilon^k \to 0$, with $\varepsilon^0 = \varepsilon$.
Further, for a sufficiently large $k$ there exist a constant $\eta>0$ satisfying $2\eta > (1+\tilde \gamma )\varepsilon^k $.
This implies that $\| z^{k+1}- \bar z\| \leq 2\eta$ and from \eqref{eq:kappa_omega_rk_bound} we obtain
\begin{align}\label{eq:bound_r_active_set}
	\|r^k \| \leq (\kappa^k + {\omega \eta}) (1+\tilde \gamma )\varepsilon^k
\end{align}
It follows from~\eqref{eq:active_set_on_r} and~\eqref{eq:bound_r_active_set} that $g_i(w^k) = 0$ remains strictly active at $w^k$ if it holds that
$
(\kappa^k + {\omega \eta}) (1+\tilde \gamma )\varepsilon^k < \frac{\bar \mu_i}{L_{\varphi_i}}.
$
Because $\varepsilon^k$ is shrinking this can always be made true for a sufficiently large but finite $k$.
In particular, we need:
\begin{align}\label{eq:epsilon_cond}
	\varepsilon^{k} <  \frac{\bar\mu_i}{L_{\varphi_i} (\kappa^k + {\omega \eta}) (1+\tilde \gamma )} := \rho_i^\mu, \ \forall  i \in \A_+(\bar w,\bar \mu).
\end{align}
Similarly, all inactive constraints should stay inactive, and using similar reasoning for $i \in \A^{\mathrm{C}}(\bar{w})$ we have from $\|r \| \leq \frac{-g(\bar w)}{L_{\phi_i}}$:
\begin{align}\label{eq:epsilon_cond_inactive}
	\varepsilon^{k} <  \frac{-g_i(\bar w)}{L_{\varphi_i} (\kappa^k + {\omega \eta}) (1+\tilde \gamma )}  := \rho_i^g,\ \forall  i \in \A^{\mathrm{C}}(\bar{w}).
\end{align}
Moreover, from Theorem~\ref{th:fiacco_nlp_sensitivity}, the Lipschitz estimate on $z(r)$ is only valid for $\|r \| < \delta_1$, which combined with~\eqref{eq:epsilon_cond} yields the requirement $\varepsilon^k <  \frac{\delta_1}{(\kappa^k + {\omega \eta}) (1+\tilde \gamma )} =: \hat{\delta}$.

For large enough but finite $k$, we can make sure both \eqref{eq:epsilon_cond} and \eqref{eq:epsilon_cond_inactive} hold true.
In particular, we have that for
$
\rho^{\mathrm{s}} = \varepsilon^k < \min \Big(  \min_{ i \in \A_+(\bar{w},\bar{\mu}) } \rho_i^\mu , \min_{ i \in \A^\mathrm{C}(\bar{w})} \rho^g_i , \hat{\delta} \Big),	
$
and for $\|z^{k} - \bar{z} \| \leq \rho^{\mathrm{s}}$ it holds that  $\A_+(w^k,\mu^k) = \A_+(\bar{w},\bar{\mu})$ and
$\A^\mathrm{C}(w^k) = \A^\mathrm{C}(\bar{w})$.

On the other hand, for weakly active constraints, the switching function $\varphi_i(r) = \mu_i(r) + g_i(w(r))$ may change its sign arbitrarily close to $r=0$, and hence such constraints may be either active or inactive at $w^{k}$.
From this, the conclusion of the theorem and \eqref{eq:active_set_stabilization} follow.
Under strict complementarity $\A_0(\bar{w},\bar{\mu}) = \emptyset$ we obtain from \eqref{eq:active_set_stabilization} that 	$\mathcal{A}(w^k)=\mathcal{A}(\bar w)$, and the proof is complete.
\qed 

Note that, depending on the values of $\mu_i$ and $g_i(\bar w)$, no a priori relation between stabilization radius $\rho^{\mathrm{s}}$ and SQP local convergence  radius $\varepsilon$ can be guaranteed.
Consequently, full active-set stabilization may occur immediately, intermediately or only asymptotically, and weakly active constraints may become active only in the limit.

\begin{example}(Active-set stabilization)
	Consider the convex NLP from~\cite[Example 4.15]{Izmailov2014}:
	\begin{align*}
		\min_{w\in \R}\; w^2 + w^4
		\quad \mathrm{s.t.} \quad w\geq 0.
	\end{align*}
	The unique local minimizer is $\bar w = 0$ with multiplier $\bar \mu = 0$, and it satisfies LICQ and SSOSC.
	If the SQP method is initialized at a sufficiently small $w^0 > 0$, then the iterates converge quadratically to $\bar w = 0$, but satisfy $w^{k+1}= \frac{4(w^k)^3}{6(w^k)^2+1} > 0$ for all $k$, and identify the weakly active constraint only in the limit (with $\A(w^k) = \emptyset$ for all $k$,  and $\A(\bar w) = \{1\}$), as predicted by the estimate~\eqref{eq:active_set_stabilization}.
	If the objective is modified to $(w+1)^2 + (w+1)^4$, the solution remains $\bar w = 0$, but now $\bar \mu = 6$, so strict complementarity holds.
	In line with Proposition~\ref{th:active_set_stabilization}, it can be verified that in this case the SQP method identifies the active set in finitely many iterations.
\end{example}

\section{Mathematical programs with complementarity constraints}\label{sec:mpcc_theory}
The MPCC~\eqref{eq:mpcc_intro} can be reformulated into a standard nonlinear program (NLP), by replacing \eqref{eq:mpcc_intro_comp} by a set of inequality constraints in~\eqref{eq:mpcc_nlp_compG}--\eqref{eq:mpcc_nlp_bilinear}:
\begin{mini!}[2]
	{\substack{w\in \R^{n}}}{f(w)\label{eq:mpcc_nlp_obj}}
	{\label{eq:mpcc_nlp}}{}
	\addConstraint{h(w)}{=0 \label{eq:mpcc_nlp_eq}}
	\addConstraint{g(w)}{\leq0 \label{eq:mpcc_nlp_ineq}}
	\addConstraint{G(w)}{\geq0 \label{eq:mpcc_nlp_compG}}
	\addConstraint{H(w)}{\geq0 \label{eq:mpcc_nlp_compH}}
	\addConstraint{G_i(w)H_i(w)}{\leq0, \quad i=1,\ldots,m. \label{eq:mpcc_nlp_bilinear}}
\end{mini!}
The nonnegativity of $G(w)$ and $H(w)$ allow a few other equivalent NLP reformulations, e.g., \eqref{eq:mpcc_nlp_bilinear} can be replaced by $G_i(w)H_i(w) = 0$, or an aggregated formulation with $G(w)^\top H(w) \leq 0$ or $G(w)^\top H(w) = 0$, but none of them avoids the degeneracy of NLP reformulations.
In particular,  most of the standard NLP theory cannot be applied directly, since the constraints \eqref{eq:mpcc_nlp_compG}-\eqref{eq:mpcc_nlp_bilinear} lead to violation of the MFCQ (and thus SMFCQ and LICQ) at all feasible points.
This means the set of Lagrange multipliers is necessarily unbounded, complicating computational and theoretical considerations for this problem class.
Instead, one usually regards tailored MPCC notions and optimality conditions.
They all rely on the piecewise nature of the MPCC's feasible set.

\subsection{First-order optimality conditions}
Algebraic conditions for stationarity of regular NLPs are given by the KKT conditions~\eqref{eq:kkt}. 
For MPCCs, there are more nuances, and there are several first-order stationarity concepts, which we review briefly here. 
For surveys see~\cite{Kim2020,Scheel2000}.
Unlike the standard KKT conditions, some of the MPCC-multiplier-based stationarity concept may allow first-order descent directions~\cite{Leyffer2007,Scheel2000}.

We call an NLP regular if it satisfies some first- and second-order regularity conditions, e.g., LICQ and SSOSC.
Most of the MPCC theory relies on defining auxiliary regular NLPs related to the MPCC.
For this purpose, we define the following index sets which depend on a feasible point $w$ of the MPCC~\eqref{eq:mpcc_intro}:
\begin{align*}
	\I_{0+}(w) &= \{i \in \{1, \ldots, m \} \mid G_i(w)=0, H_i(w)>0\},\\
	\I_{+0}(w) &=	\{i \in \{1, \ldots, m\} \mid G_i(w)>0, H_i(w)=0\},\\
	\I_{00}(w) &= \{i \in \{1, \ldots, m \} \mid G_i(w)=0, H_i(w)=0\}.
\end{align*}

If clear from the context, we omit the argument $w$ in the index sets.
For the standard inequality constraints $g(w) \leq 0$, we define the active sets as in~\eqref{eq:active_sets}--\eqref{eq:active_sets_strict}.

We call indices $i \in \I_{00}(w)$ degenerate index pairs, and indices $i \in \{1,\ldots,m\} \setminus \I_{00}(w)$ nondegenerate.
Most theoretical difficulties arise in the presence of degenerate indices.
If $\I_{00}(w) = \emptyset$, then the lower-level strict complementarity (LLSC) condition is said to hold. 
This assumption is usually regarded as too restrictive and is therefore almost never imposed in theoretical considerations~\cite{Scheel2000}.

To compactly define branch NLPs below, we first define a few more index sets.
Let $\tilde{\mathcal{P}}(w) = \{ \tilde \I(w)  \subseteq  \I_{00}(w)\}$ denote the powerset of the degenerate indices.
Then, at a feasible point $w\in \Omega$, we define the set of index sets: 
\begin{align*}
	\mathcal{P}(w)  = \{ \I \subseteq \{1,\ldots,m\}\mid \exists\, \tilde{\I}  \in \tilde{\mathcal{P}}(w),\, \I = \I_{0+}(w)\cup \tilde{\I}\}.
\end{align*}
Moreover, for a given $\I \in \mathcal{P}(w)$ we define its complement $\I^\mathrm{C} = \{1,\ldots,m\} \setminus \I$.
Clearly, it holds that $\I^\mathrm{C} \subseteq \I_{+0}(w)\cup \I_{00}(w)$.
It can be seen that the number of sets in $\mathcal{P}(w)$ is equal to $N = 2^{c}$, where $c$ is the number of degenerate indices.

Let $\bar w$ be a feasible point of the MPCC~\eqref{eq:mpcc_intro}.
Using the partitions $\I \in \mathcal P(\bar w)$, we define the $N$ so-called branch NLPs, denoted by $\BNLP_\I$, obtained by fixing the corresponding branches of the L-shaped complementarity set:
\begin{mini!}[2]
	{\substack{w\in \R^{n}}}{f(w)\label{eq:bnlp_obj}}
	{\label{eq:bnlp}}{}
	\addConstraint{h(w)}{=0 \label{eq:bnlp_eq}}
	\addConstraint{g(w)}{\leq0 \label{eq:bnlp_ineq}}
	\addConstraint{G_i(w) =0,\ H_i(w) }{\geq 0,\quad \forall i \in \I(\bar{w}) \label{eq:bnlp_comp1}}
	\addConstraint{G_i(w) \geq 0,\ H_i(w)} {= 0,\quad \forall i \in \I^\mathrm{C}(\bar{w}). \label{eq:bnlp_comp2}}
\end{mini!}
Another useful NLP associated to the MPCC~\eqref{eq:mpcc_intro} is the so-called relaxed NLP (RNLP), which reads as:
\begin{mini!}[2]
	{\substack{w\in \R^{n}}}{f(w)\label{eq:rnlp_obj}}
	{\label{eq:rnlp}}{}
	\addConstraint{h(w)}{=0 \label{eq:rnlp_eq}}
	\addConstraint{g(w)}{\leq0 \label{eq:rnlp_ineq}}
	\addConstraint{G_i(w) =0,\ H_i(w) }{\geq 0,\quad \forall i \in \I_{0+}(\bar{w}) \label{eq:rnlp_comp1}}
	\addConstraint{G_i(w) \geq 0,\ H_i(w)} {= 0,\quad \forall i \in \I_{+0}(\bar{w}) \label{eq:rnlp_comp2}}
	\addConstraint{G_i(w) \geq 0,\ H_i(w)} {\geq 0,\quad \forall i \in \I_{00}(\bar{w}). \label{eq:rnlp_bi}}
\end{mini!} 
Observe that if LLSC holds, i.e., $\I_{00}(\bar{w}) = \emptyset$, then there is only one BNLP and it coincides with the RNLP.
Clearly, the feasible set of each $\BNLP_\I$ is contained in the feasible set of the MPCC, which in turn is contained in the feasible set of the RNLP.
Consequently, if $\bar{w}$ is a local minimizer of the RNLP, then it is a local minimizer of the MPCC. 
The converse is not true. 
The point $\bar{w}$ is a local minimizer of the MPCC if and only if it is a local minimizer of every $\BNLP_\I$~\cite{Luo1996,Scheel2000}. 
In view of this, these auxiliary NLPs provide a natural framework for defining tailored MPCC notions, some of which we list next.

\begin{definition}
	The MPCC--Lagrangian is the \textit{standard} Lagrangian for the RNLP, and reads as:
	\begin{align*}
		\mathcal{L}(w,\lambda,\mu,\xi,\nu) = f(w) + \lambda^\top h(w) + \mu^\top g(w)- \xi^\top G(w)- \nu^\top H(w),
	\end{align*}
	with the MPCC--Lagrange multipliers $\lambda \in \R^{m_h}$, $\mu \in \R^{m_g}$, $\xi \in \R^{m}$ and $\nu \in \R^{m}$.
\end{definition}
Note that the MPCC--Lagrangian $\mathcal L(\cdot)$ differs from the standard Lagrangian $L(\cdot)$ for the NLP reformulation of the MPCC \eqref{eq:mpcc_nlp}, by omitting the terms $G_i(w)H_i(w)$ and their multipliers.
Observe that the MPCC--Lagrangian coincides with the usual Lagrangian of the $\BNLP_\I$ and is independent of the index~$\I$.

\begin{definition}
	The MPCC \eqref{eq:mpcc_intro} is said to satisfy MPCC--LICQ, MPCC--SMFCQ, or MPCC--MFCQ at a feasible point $\bar w$ if the corresponding RNLP associated with $\bar w$ satisfies, respectively, the LICQ, the SMFCQ (with the corresponding MPCC--Lagrange multipliers), or the MFCQ at the same point $\bar w$.
\end{definition}

By applying the KKT conditions to the RNLP~\eqref{eq:rnlp}, we obtain the so-called strong stationarity (S-stationarity) concept.
In particular, a feasible point $\bar{w}$ is called S-stationary if there exist MPCC--Lagrange multipliers $\bar{\lambda},\bar{\mu},\bar{\xi}$ and $\bar{\nu}$ such that:
\begin{align*}
	&\nabla_{w} \mathcal{L}(\bar{w},\bar{\lambda},\bar{\mu},\bar{\xi}, \bar{\nu}) = 0,\\
	&h(\bar{w}) = 0,\\
	&0 \leq \bar{\mu} \perp -g(\bar{w}) \geq 0,\\
	& G_i(\bar{w}) = 0,\; \bar{\xi}_i \in \R,\; H_i(\bar{w}) \geq0,\; \bar{\nu}_i = 0,\;  & \forall i \in 	\I_{0+}(\bar{w}),\\
	& G_i(\bar{w}) \geq 0,\; \bar{\xi}_i = 0,\; H_i(\bar{w}) = 0,\; \bar{\nu}_i \in  \R,\;  & \forall i \in 	\I_{+0}(\bar{w}),\\
	& G_i(\bar{w}) =0,\; \bar \xi_i \geq 0, \; & \forall i \in \mathcal{I}_{00}(\bar{w}),\\
	& H_i(\bar{w}) =0,\; \bar \nu_i \geq 0, \; &\forall i \in 	\mathcal{I}_{00}(\bar{w}).
\end{align*}
Note that the last two conditions originate from the complementarity slackness, but are reduced since active sets are known at $\bar w$.
Weaker stationarity concepts are derived from alternative auxiliary NLPs and by employing tools from nonsmooth calculus~\cite{Outrata1999,Scheel2000}.
They are characterized by relaxed sign restrictions on $\bar\xi_i$ and $\bar\nu_i$ for $i \in \I_{00}(\bar w)$, and read as follows:

\begin{itemize}
	\item Weak Stationarity (W-stationarity): if $\bar{\xi}_i \in \R$ and $\bar{\nu}_i \in \R$ for all $i \in \mathcal{I}_{00}(\bar{w})$.
	\item Abadie Stationarity (A-stationarity): if $\bar{\xi}_i \geq 0$ or $\bar{\nu}_i \geq0$ for all $i \in \mathcal{I}_{00}(\bar{w})$.
	\item Clarke Stationarity (C-stationarity):  if $\bar{\xi}_i\bar{\nu}_i \geq0$ for all $i \in \mathcal{I}_{00}(\bar{w})$.
	\item Mordukhovich Stationarity (M-stationarity):  if either $\bar{\xi}_i >0$ and $\bar{\nu}_i >0$ or $\bar{\xi}_i\bar{\nu}_i =0$ for all $i \in \mathcal{I}_{00}(\bar{w})$.
\end{itemize}
The stationarity concepts are ordered according to decreasing strength as follows: $\mathrm{S}$, $\mathrm{M}$, $\mathrm{C}$, $\mathrm{A}$, and $\mathrm{W}$.
Observe that if LLSC holds, then all multiplier-based stationarity concepts coincide and they reduce to S-stationarity.
Except for S-stationarity, all other stationarity concepts may admit first-order descent directions~\cite{Leyffer2007,Scheel2000}, even if MPCC--LICQ holds.
A stationarity concept, which does not suffer from this is the so-called B-stationarity, which we recall next.

Let $\Omega$ denote the feasible set of the MPCC~\eqref{eq:mpcc_intro}.
Denote by $\mathcal{T}_{\Omega}(w)$ the tangent cone at $w\in\Omega$, cf.~\cite[Definition 12.2]{Nocedal2006}.
A first-order necessary condition for locally optimality reads as: 
if a point $\bar{w}$ is locally optimal, then there is no feasible first-order descent direction,
that is, $ \nabla f(\bar{w})^\top d \geq 0\; \text{ for all } d \in \mathcal{T}_{\Omega}(\bar{w})$.
A point satisfying this variational inequality is called B-stationary.
For regular NLPs (e.g., satisfying LICQ), B-stationarity is simply equivalent to the KKT conditions. 
For MPCCs, there also exist purely algebraic but combinatorial characterizations of B-stationarity~\cite{Scheel2000}.
Here, we consider a useful characterization due to Hu and Ralph~\cite{Hu2002}, formulated in terms of BNLPs in the following definition.
\begin{definition}\label{def:b_stat_bnlp}
	A feasible point $\bar{w} \in \Omega$ of the MPCC~\eqref{eq:mpcc_intro} is a piecewise stationary or B-stationary point of the MPCC~\eqref{eq:mpcc_intro} if $\bar{w}$ is a stationary point of the $\BNLP_\I$ for each $\I \in \mathcal{P}(\bar{w})$, that is for each partition, there exists Lagrange multipliers $\bar{\lambda}^{\I},\, \bar{\mu}^{\I},\, \bar{\xi}^{\I}$ and $\bar{\nu}^{\I}$ such that:
	\begin{subequations}\label{eq:kkt_bnlp}
		\begin{align}
			&\nabla_{w} \mathcal{L}(\bar{w},\bar{\lambda}^\I,\bar{\mu}^\I,\bar{\nu}^\I,\bar{\xi}^\I) = 0,\\
			&h(\bar{w}) = 0,\\
			&0 \leq \bar{\mu}^\I \perp -g(\bar{w}) \geq 0,\\
			&G_i(\bar w) = 0,  & &\forall i \in \I,\\
			&H_i (\bar w) = 0  & &\forall i \in \I^{\mathrm{C}},\\
			&0 \leq  H_i(\bar{w}) \perp \bar{\nu}_i^\I \geq 0,  & &\forall i \in \I,\\
			&0 \leq G_i(\bar{w}) \perp \bar{\xi}_i^\I \geq 0,   & &\forall i \in \I^{\mathrm{C}}.
		\end{align}
	\end{subequations}
\end{definition}
S-stationarity is the only stationarity concept that can be shown to be equivalent to B-stationarity in a straightforward manner; see~\cite[Theorem~4]{Scheel2000}. 
If $\bar{w}$ is an S-stationary point of the MPCC~\eqref{eq:mpcc_intro}, then it is also B-stationary. 
Conversely, if the MPCC--LICQ holds at $\bar{w}$, then every B-stationary point is S-stationary. 
The MPCC--LICQ condition cannot be relaxed; if it fails to hold, there may exist B-stationary points that are not S-stationary. 
The next weaker constraint qualification, MPCC-MFCQ, is already insufficient to guarantee S-stationarity; for example, under MPCC--MFCQ a B-stationary point may be only M-stationary, see~\cite[Example~3]{Scheel2000}.

Additionally, if $\bar w$ is a B-stationary point at which MPCC--LICQ holds, then 
for each $\I \in \mathcal P(\bar w)$ there exists a unique multiplier vector $(\bar\lambda^{\I}, \bar\mu^{\I}, \bar\xi^{\I}, \bar\nu^{\I})$ associated with $\bar w$ and the corresponding $\BNLP_\I$; since the gradient of the MPCC--Lagrangian is independent of $\I$ in~\eqref{eq:kkt_bnlp}, and active sets coincide for all $\BNLP_\I$ at $\bar w$, these multipliers are also independent of $\I$, that is, they are identical for all branch NLPs.
Notably, this avoids the combinatorial complexity of verifying B-stationarity, since it is not necessary to check \eqref{eq:kkt_bnlp} for all $\BNLP_\I$.
We highlight this fact in the following proposition, which will be useful in the analysis below.
\begin{proposition}(\cite[Proposition 4.3.5 ]{Luo1996})\label{prop:uniqe_mpcc_multipliers}
	Let $\bar w\in\Omega$ be a B-stationary point of the MPCC \eqref{eq:mpcc_intro} at which the MPCC--LICQ holds. 
	Then the multiplier vector $(\bar\lambda,\bar\mu,\bar\xi,\bar\nu)$ satisfying the KKT system \eqref{eq:kkt_bnlp} is unique. 
	In particular, if $(\bar\lambda^{\I},\bar\mu^{I},\bar\xi^{\I},\bar\nu^{\I})$ is any multiplier quadruple satisfying \eqref{eq:kkt_bnlp} for some index $\I \in \mathcal{P}(\bar w)$, then $(\bar\lambda^{\I},\bar\mu^{\I},\bar\xi^{\I},\bar\nu^{\I}) = (\bar\lambda,\bar\mu,\bar\xi,\bar\nu)$.
\end{proposition}

In summary, B-stationarity is the tightest stationarity concept for MPCCs, and under MPCC--LICQ  (and the weaker MPCC--SMFCQ~\cite{Scheel2000}) it is equivalent to S-stationarity, which is purely algebraic and does not require any combinatorial exploration of all BNLPs. 
This is the setting considered in this paper.

\subsection{Second-order optimality conditions}
The second-order sufficient conditions for MPCCs are also defined in terms of the RNLP and BNLPs.
The MPCC--SOSCs and the stronger MPCC--SSOSC are formulated at an S-stationary point so that a single set of MPCC--Lagrange multipliers is valid for all BNLPs (cf. Proposition~\ref{prop:uniqe_mpcc_multipliers}).
The so-called MPCC critical cone is obtained as the union of the usual critical cones of all $\BNLP_{\I}$ at $\bar w$~\cite{Luo1996}.
Analogous to the standard NLP case, the MPCC--SOSC  can be strengthened by enlarging the critical cone via dropping the inequalities corresponding to zero multipliers, and we define the MPCC strong critical cone~\cite{Luo1996,Scheel2000}:

\begin{align*}
	\mathcal{C}_{\mathrm{MPCC}}^{\mathrm{S}}(\bar{w}) := \bigcup_{\I \in \mathcal P(\bar w)}
	\Big\{ d\in\R^n \ \Big|\
	& \nabla h(\bar{w})^\top d = 0,\ \nabla g_{i}(\bar{w})^\top d = 0, \ \forall i \in \A_+(\bar{w},\bar{\mu}),\\ 
	& \nabla G_{i}(\bar{w})^\top d = 0, \ \forall i \in \I \cup \{ j \in \I^{\mathrm{C}} \mid \bar \xi_j >0 \} ,\\ 
	& \nabla H_{i}(\bar{w})^\top d = 0, \ \forall i \in \I^\mathrm{C} \cup \{ j \in \I \mid \bar \nu_j >0 \}
	\Big\}. 
\end{align*}
\begin{definition}[MPCC--SSOSC]
	The MPCC--SSOSC (or the piecewise SSOSC) holds at S-stationary point $\bar{w}$ with MPCC--Lagrange multipliers
	$(\bar\lambda,\bar\mu,\bar\nu,\bar\xi)$ if
	\[
	d^\top \nabla_{ww}^2 \mathcal{L}(\bar{w},\bar\lambda,\bar\mu,\bar\xi,\bar\nu)\, d>0
	\quad \forall  d\in \mathcal{C}^\mathrm{S}_{\mathrm{MPCC}}(\bar{w})\setminus\{0\}.
	\]
\end{definition}
The piecewise (S)SOSC implies the standard (S)SOSC for each BNLP.
This also implies that $\bar{w}$ is a strict local solution for each BNLP, and thus for the MPCC~\cite{Hu2002}.
It is sufficient for local optimality of an S-stationary point.
A stronger condition is the so-called RNLP-(S)SOSC, which is the usual (S)SOSC of the RNLP~\eqref{eq:rnlp}.

We recall the MPCC analogue of strict complementarity, referred to as upper-level strict complementarity (ULSC)~\cite{Scheel2000}.
In addition, in this paper, we introduce and use a new condition that is weaker than the ULSC, the partial ULSC (PULSC).
\begin{definition}[ULSC]\label{def:ulsc}
	Let $\bar w$ be an S-stationary point of the MPCC~\eqref{eq:mpcc_intro}, with associated MPCC--Lagrange multipliers $(\bar \lambda, \bar \mu, \bar \xi, \bar \nu)$.
	The upper-level strict complementarity (ULSC) condition holds if, for every $i \in \I_{00}(\bar w)$, the multipliers satisfy the condition:
	$\bar \xi_i > 0$ and $\bar \nu_i > 0$.
\end{definition}
\begin{definition}[PULSC]\label{def:pulsc}
	Let $\bar w$ be an S-stationary point of the MPCC~\eqref{eq:mpcc_intro}, with associated multipliers	$(\bar \lambda, \bar \mu, \bar \xi, \bar \nu)$.
	The partial ULSC (PULSC) condition holds if, for every $i \in \I_{00}(\bar w)$, at least one of the inequalities
	$\bar \xi_i > 0$ or $\bar \nu_i > 0$ holds.
\end{definition}

\section{Sequential quadratic programming with complementarity constraints (SQPCC)}~\label{sec:sqpcc}
In Section~\ref{sec:sqpcc_derivation}, we recall the derivation of the SQPCC method.
Section~\ref{sec:sqpcc_convergence} presents our new local convergence result.
Furthermore, in Section~\ref{sec:leyffer}, we revisit a counterexample from~\cite{Leyffer2007} and show that our result is not affected by it.
In addition, we discuss some interesting properties of the SQPCC method and compare it with the usual SQP method.
Section~\ref{sec:sqpcc_tight} tightens the convergence result to exclude spurious iterates.
Section~\ref{sec:sqpcc_active_set} provides an active-set stabilization result.

\subsection{Method statement}\label{sec:sqpcc_derivation}
The SQP method for standard NLPs is derived from a piecewise linearization of the KKT conditions.
In the absence of appropriate constraint qualifications, the KKT conditions may fail to be necessary for local optimality of the MPCC~\eqref{eq:mpcc_intro}; moreover, the associated Lagrange multipliers are necessarily unbounded, and the resulting linearization may become inconsistent arbitrarily close to a solution; see~\cite{Fletcher2006}.
To avoid these complications, as proposed by~\cite[Section 5]{Scholtes2004}, one can instead consider a piecewise linearization of the B-stationarity conditions, given in Eq.~\eqref{eq:kkt_bnlp} from Def.~\ref{def:b_stat_bnlp}, at the linearization point $(w^k,\lambda^k,\mu^k,\xi^k,\nu^k)$, which yields:
\begin{subequations}\label{eq:kkt_bnlp_pwl}
	\begin{align*}
		\begin{split}
			&\nabla f(w^k) + \nabla h(w^k) \lambda^{k+1} + \nabla g(w^k) \mu^{k+1}  \\
			&- \nabla G(w^k) \nu^{k+1}   -  \nabla H(w^k) \xi^{k+1}  +
			\nabla_{ww}^2 \mathcal{L}(w^k,\lambda^k,\mu^k,\xi^k,\nu^k) \Delta w^k = 0,	\\
		\end{split}\\
		&h(w^k) + \nabla h(w^k)^\top \Delta w^k = 0,\\
		&0 \leq -g(w^k) -\nabla g(w^k)^\top \Delta w^k \perp \mu^{k+1} \geq 0,\\
		& G_i(w^{k}) +  \nabla G(w^k)^\top \Delta w^k = 0, \;  \quad \forall i \in 	\I(w^{k}),\\
		&  H_i(w^k) +  \nabla H(w^k)^\top \Delta w^k = 0,\;  \quad \forall i \in 	\I^{\mathrm{C}}(w^{k}),\\
		& 0 \leq  H_i(w^k) +  \nabla H(w^k)^\top \Delta w^k \perp {\xi}_i^{k+1} \geq 0,\;  \quad \forall i \in 	\I(w^{k}),\\
		& 0 \leq G_i(w^{k}) +  \nabla G(w^k)^\top \Delta w^k \ \perp {\nu}_i^{k+1} \geq 0 \;  \quad \forall i \in 	\I^{\mathrm{C}}(w^{k}).
	\end{align*}
\end{subequations}
Clearly, if the linearization point coincides with a solution of the MPCC, then $\Delta w^{k} = 0$ and the B-stationarity conditions~\eqref{eq:kkt_bnlp} are recovered.
Otherwise, the resulting conditions correspond to the B-stationarity conditions of a quadratic program with complementarity constraints (QPCC):
\begin{mini!}[2]
	{\substack{\Delta w \in \R^{n}}}{\nabla f(w^k)^\top \Delta w + \frac{1}{2} \Delta w^\top \nabla_{ww}^2 \mathcal{L}(w^k,\lambda^k,\mu^k,\xi^k,\nu^k) \Delta w \label{eq:qpcc_obj}}
	{\label{eq:qpcc}}{}
	\addConstraint{h(w^k) + \nabla h(w^k)^\top \Delta w}{=0 \label{eq:qpcc_eq}}
	\addConstraint{g(w^k) +\nabla g(w^k)^\top \Delta w}{\leq0 \label{eq:qpcc_ineq}}
	\addConstraint{0 \leq G(w^k) + \nabla G(w^k)^\top \Delta w \perp H(w^k) + \nabla H(w^k)^\top \Delta w}{\geq 0 \label{eq:qpcc_comp}.}
\end{mini!}
Using QPCC subproblems, the SQP method can be generalized by solving a sequence of quadratic programs with complementarity constraints (SQPCC).  
This SQPCC approach was first proposed in~\cite{Scholtes2004} in a more general setting, namely for star-shaped nonconvex geometric constraints, to which the L-shaped complementarity set belongs.
In contrast to the standard SQP method, the subproblems explicitly linearize the complementarity constraints while preserving their combinatorial structure.
This leads to a more natural quadratic approximation of the MPCC.

Observe that the QPCC uses the Hessian of the MPCC Lagrangian $\mathcal L(\cdot)$ instead of the standard Lagrangian $L(\cdot)$ of~\eqref{eq:mpcc_nlp}.
As in the SQP method, one may use a positive definite Hessian approximation $H^k$ in place of the exact Hessian.

Several local algorithms for solving QPCCs can, under suitable assumptions such as MPCC--LICQ and positive definiteness of the Hessian $H^k$, efficiently compute an S-stationary point of a QPCC~\cite{Hall2024,Luo1996,Pozharskiy2026,Scholtes2004}. 
We will assume this in the convergence proof below.
Note that finding a global optimum of a QPCC is sufficient for finding an S-stationary point of the QPCC, but it is not necessary for the convergence of the SQPCC method.

Note that satisfaction of the linearized complementarity conditions~\eqref{eq:qpcc_comp} at an iterate $w^{k}$ does not, in general, imply satisfaction of the nonlinear complementarity conditions, that is, $G_i(w^k) H_i(w^k) \neq 0$.  
We therefore define the complementarity index sets for the QPCC~\eqref{eq:qpcc} using the superscript ``$\mathrm{qp}$'' and write $(\I_{0+}^{\mathrm{qp}}(w^k), \I_{+0}^{\mathrm{qp}}(w^k), \I_{00}^{\mathrm{qp}}(w^k))$.
In the special case where the complementarity functions $G$ and $H$ are affine in $w$, one has $G_i(w^k) H_i(w^k) = 0$ for all $k>0$.  
That is, the iterates are complementarity-feasible and the superscript ``$\mathrm{qp}$'' can be omitted.

\subsection{Local convergence of the SQPCC method}\label{sec:sqpcc_convergence}

A local convergence analysis of the SQPCC method was outlined in~\cite[Section 5]{Scholtes2004}. 
However, a detailed convergence theory was not the main focus there; rather, the emphasis was on introducing the method and some basic notions.
Thus, the analysis given in~\cite{Scholtes2004} relies on rather restrictive assumptions, namely strict complementarity for the inequality constraints, ULSC for the complementarity constraints, an RNLP--SSOSC condition (which is stronger than the MPCC--SSOSC assumed below), and initialization sufficiently close to the solution so that the active set is already stabilized. 
A stronger version of this result, requiring only MPCC--SSOSC and no strict complementarity assumption, is established in Part~I of the proof below. 
Parts~II and~III further remove the ULSC condition and enlarge the local region of convergence, thereby yielding a more general and stronger convergence result.

Our proof relies on the observation that the SQP method and Theorem~\ref{th:sqp_local_convergence} can be applied to each ${\BNLP}_{\I}$. 
A similar decomposition idea appears in piecewise SQP (PSQP) methods~\cite[Theorem~6.4.3]{Luo1996}.

Just like in the PSQP method, due to the nonconvexity of the QPCC subproblems, one cannot expect a unique convergence sequence.
However, this nonuniqueness is not problematic for the present result, which is formulated as an existence statement,  see also the discussion in Sec.~\ref{sec:leyffer}.
For the sake of generality, we do not yet impose additional assumptions in Theorem~\ref{th:sqpcc_local_convergence}, but later provide conditions under which each sequence converges (cf. Corollary~\ref{cor:no_spurious_sequences}) and under which the converging sequence is locally unique.

Note that PSQP methods still solve QPs and rely on explicit active-set identification strategies (as discussed in Sec.~\ref{sec:intro}), which impose additional restrictions on the initialization.
In contrast, in the SQPCC approach, the active set selection is done implicitly by the QPCC solver, allowing to relax the restrictive active-set identification and strict complementarity assumptions.

The main idea of the proof is to show that, if initialized sufficiently close to a solution, the solutions of the QP subproblems solved by SQP when applied to ${\BNLP}_{\I}$ are S-stationary points of the QPCC~\eqref{eq:qpcc}.  
As a consequence, convergence of the SQPCC algorithm follows directly from the local convergence theory of SQP.  
Interestingly, we show that weaker non-S-stationary points $\tilde w$ in a neighborhood of an S-stationary point $\bar w$ can attract iterates, which subsequently converge to an S-stationary point.  
This phenomenon further enlarges the local convergence region of the SQPCC method (Part~III of the proof below), a property that plain SQP~\cite{Fletcher2006} and PSQP methods of~\cite{Luo1996} do not possess.  
This, and the typical convergence behavior is illustrated in the example below.
\begin{figure}[t]
	\centering
	\includegraphics[width=0.9\textwidth]{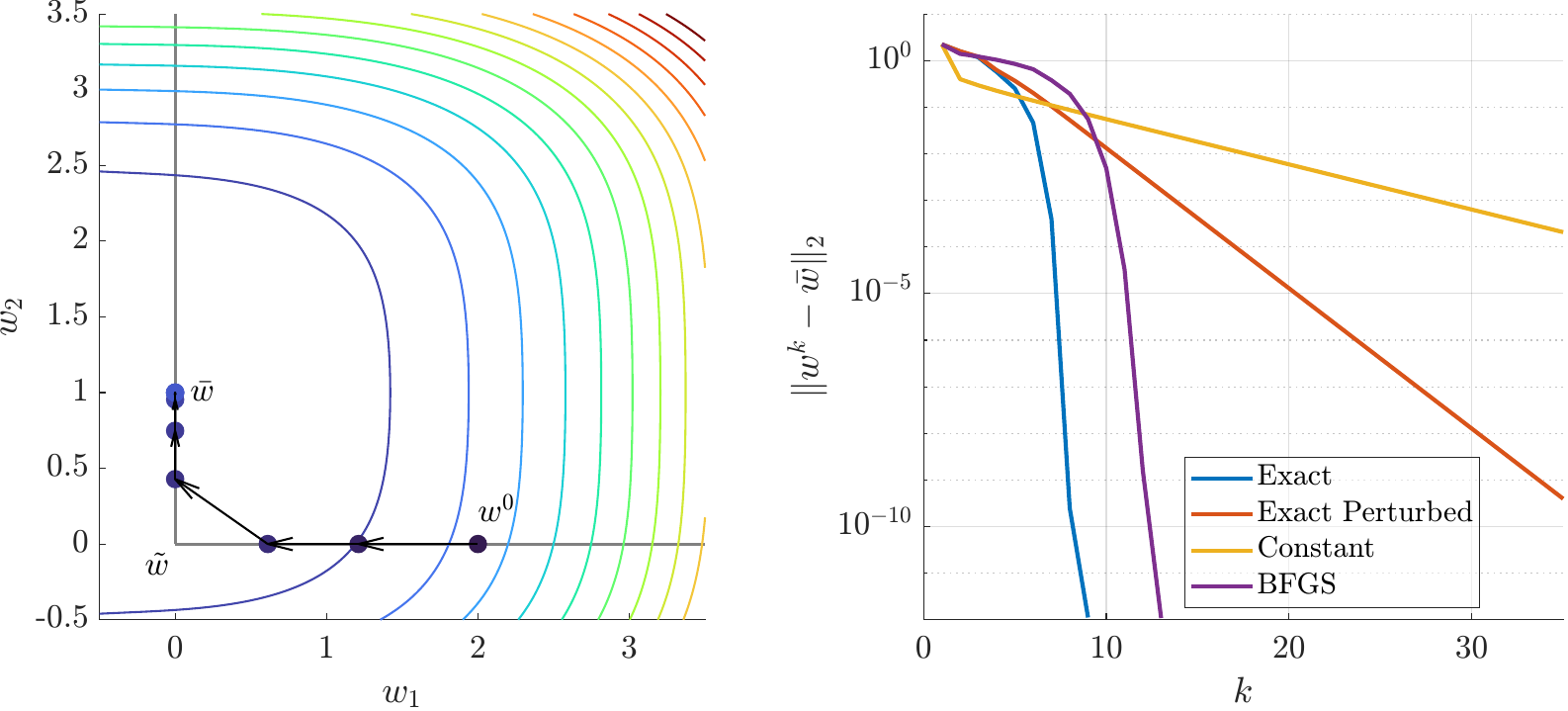}
	\caption{
		The left plot illustrates the iterates of the exact-Hessian SQPCC method in the $(w_1,w_2)$-plane.  
		The right plot illustrates the errors over the iterations for different Hessian approximations, showing the corresponding convergence rates.
	}
	\label{fig:sqpcc_local_convergence}
	\vspace{-0.3cm}
\end{figure}

\begin{example}\label{ex:sqpcc_local_convergence}
	We illustrate the local convergence of the SQPCC method with the following example:
	\begin{align*}
		\min_{w\in \R^2}\; w_1 + w_1^2 + w_1^3 + (w_2-1)^4 + (w_2-1)^2
		\quad \mathrm{s.t.} \quad
		0 \leq w_1 \perp w_2 \geq 0.
	\end{align*}
	The problem has a unique local minimizer $\bar w = (0,1)$ with $f(\bar w) = 0$, which is an S-stationary point.  
	Moreover, the origin $\tilde w = (0,0)$ with $f(\tilde w) = 2$ is not a local minimizer, but a C-stationary point with MPCC--Lagrange multipliers $\tilde \xi = -1$ and $\tilde \nu = -6$.  
	We use the initial point $w^0 = (2,0)$ and show the iterates of the exact-Hessian SQPCC method in the left plot of Fig.~\ref{fig:sqpcc_local_convergence}.
	The first three iterates are of the form $w^k = (w_1^k,0)$ with $w_1^k > 0$, which are S-stationary points of the QPCC since there are no degenerate indices, i.e., $\I_{00}^\mathrm{qp}(w^k) = \emptyset$.  
	However, the origin $\tilde w$ is not an S-stationary point, and the SQPCC iterates turn around the corner and converge toward $\bar w$, with iterates of the form $w^k = (0,w_2^k)$ and $w_2^k > 0$.  
	This illustrates how a C-stationary point, together with its local convergence radius where the QPCC subproblems yield S-stationary points (Part~III of the proof), attract iterates into the local convergence region of the BNLP with $w_1 = 0$, $w_2 \ge 0$ (Parts~I–II of the proof).
	
	The right plot in Fig.~\ref{fig:sqpcc_local_convergence} illustrates different convergence rates.  
	The exact Hessian yields quadratic convergence, as expected; the BFGS approximation gives superlinear convergence.  
	A perturbed exact Hessian $H^k = \nabla_{ww}^2 \mathcal L^k + I$, where $I$ is the identity, yields linear convergence since $\kappa^k > 0$ for all $k$.  
	Similarly, a constant Hessian $H^k = \mathrm{diag}(5,10)$ also gives linear convergence, with a larger $\kappa^k$ and thus a slower local convergence rate $\alpha^ k < 1$.
\end{example}

To apply SQP arguments to each BNLP, analogous to~\eqref{eq:kkt_ge_definition}, we define the problem function mapping for the BNLPs:
\begin{align*}
	\Psi(z) := \Big(\nabla_{w} \mathcal L(w,\lambda,\mu,\xi,\nu),\, h(w),\, -g(w),\, G(w),\, H(w)\Big).
\end{align*}
With a slight overloading of notation, now $z = (w,\lambda,\mu,\xi,\nu)$ collects all primal–dual variables of the MPCC.
The ``kappa'' and ``omega'' conditions are formulated for $\nabla \Psi(z^k)^\top$ and its approximation $\nabla \tilde \Psi(z^{k})^\top$, where $H^k$ replaces the exact Hessian of $\mathcal{L}(\cdot)$ in the QPCC~\eqref{eq:qpcc}.
In the line of Remark~\ref{rem:adjoint_sqp}, other problem data may also be approximated with only minor changes in the argument; for ease of exposition, we focus only on Hessian approximations.

\begin{theorem}[SQPCC local convergence]\label{th:sqpcc_local_convergence}
	Let $f:\R^n\to\R$, $h:\R^n\to\R^{m_h}$, $g:\R^n\to\R^{m_g}$, $H:\R^n \to \R^{m}$, and $G:\R^{n} \to \R^{m}$ be twice differentiable in a neighborhood of a point $\bar{w}\in\R^n$, with second derivatives continuous at $\bar{w}$. 
	Let $\bar{w}$ be a local solution of problem~\eqref{eq:mpcc_intro}, satisfying MPCC--LICQ and MPCC--SSOSC with the associated unique Lagrange multiplier $(\bar\lambda,\bar\mu,\bar \xi, \bar\nu )\in \R^{m_h}\times \R^{m_g} \times \R^{m} \times \R^{m}$.
	Further, suppose there exist constants $\gamma > 0$, $\omega > 0$ and $\bar{\kappa} < \frac{1}{3\gamma}$, and a sequence $\{\kappa^k\}$ with $\kappa^k \in [0,\bar \kappa]$ such that, for all $z^{k}$ and $z$, the following inequalities hold:
	\begin{subequations}\label{eq:kappa_omega_sqpcc}
		\begin{align}
			&\| \nabla \Psi(z)^\top - \nabla \Psi(z^k)^\top \| \leq \omega \| z- z^k \|, \label{eq:omega_cond_sqpcc}\\
			&\| \nabla \Psi(z^k)^\top - \nabla \tilde \Psi(z^k)^\top \| \le \kappa^k. \label{eq:kappa_cond_sqpcc}
		\end{align}
	\end{subequations}
	Then there exist a set $U$ and a constant $\varepsilon > 0$ such that, for any starting point $z^0 = (w^0,\lambda^0,\mu^0,\xi^0, \nu^0)\in U$ with $\B_{\varepsilon}(\bar z) \subseteq U$ and $\bar z = (\bar{w},\bar\lambda,\bar\mu,\bar\nu,\bar \xi)$, there exists a (not necessarily unique) sequence 
	$\{(w^k,\lambda^k,\mu^k,\nu^k,\xi^k)\}$ such that for $k=0,1,\dots$, the point $w^{k+1}$ is an S-stationary point of the QPCC~\eqref{eq:qpcc} with a positive definite Hessian approximation $H^k$, and $(\lambda^{k+1},\mu^{k+1},\nu^{k+1},\xi^{k+1})$ is an associated Lagrange multiplier. 
	Moreover, there exists such a sequence which converges to $(\bar{w},\bar\lambda,\bar\mu,\bar\xi,\bar\nu)$, and the rate of convergence is linear, satisfying the contraction inequality
	\begin{align}\label{eq:sqpcc_contraction_estimate}
		\| z^{k+1} - \bar z \| \le \alpha^k \| z^{k} - \bar z \| + \beta \| z^{k+1} - \bar z \|^2,
	\end{align}
	with constants $\alpha^k \in [0,1)$ for all $k$ and $\beta > 0$.
	If, in addition, $\{ \kappa^k \} \to 0$, then $\alpha^k \to 0$ and the convergence rate is superlinear; if $\bar \kappa = 0$, then $\alpha^k = 0$ and the convergence rate is quadratic.
\end{theorem}
\textit{Proof.}
The proof is divided into three parts. 
In Part~I, we apply the local SQP result from Theorem~\ref{th:sqp_local_convergence} to the branch NLPs and establish local convergence under ULSC (see Def.~\ref{def:ulsc}) for initializations at which the complementarity active sets are already identified. 
In Part~II, we show that the corresponding branch-QP solutions are, after refining the local neighborhood if necessary, S-stationary points of the QPCC also without ULSC and without prior complementarity active-set identification. 
In Part~III, we show that nonoptimal stationary points close to the solution can attract iterates even from outside the local convergence region identified in Part~II, thereby enlarging the local region from which a convergent SQPCC sequence exists.

\textbf{Part I.} 
MPCC--LICQ and MPCC--SSOSC imply the usual LICQ and SSOSC at $\bar w$ for each ${\BNLP}_{\I}$ with $\I \in \mathcal{P}(\bar w)$.  
Recall that $\bar z$ is a solution of each $\BNLP_\I$.
Furthermore, \eqref{eq:kappa_omega_sqpcc} implies that the analogous condition~\eqref{eq:kappa_omega_sqp} holds for each ${\BNLP}_{\I}$. 
Define $\gamma = \max_{\I \in \mathcal{P}(\bar w)} \gamma^\I$, where $\gamma^\I$ is the Lipschitz constant from Theorem~\ref{th:fiacco_nlp_sensitivity} applied to ${\BNLP}_{\I}$ (compare to~\eqref{eq:perturbed_nlp}).  
Therefore, Theorem~\ref{th:sqp_local_convergence} can be applied to each ${\BNLP}_{\I}$, yielding a constant $\varepsilon^\I > 0$ such that, for any starting point $z^0 = (w^0,\lambda^0,\mu^0,\nu^0,\xi^0) \in \B_{\varepsilon^{\I}}(\bar z)$, there exists a sequence $\{(w^{\I,k},\lambda^{\I,k},\mu^{\I,k},\nu^{\I,k},\xi^{\I,k})\}$ such that for $k=0,1,\dots$, the point $w^{\I,k+1}$ is a stationary point of a QP associated with ${\BNLP}_{\I}$:
\begin{mini!}[2]
	{\substack{\Delta w \in \R^{n}}}{\nabla f(w^k)^\top  \Delta w+ \frac{1}{2} \Delta w^\top \nabla_{ww}^2 \mathcal{L}(w^k,\lambda^k,\mu^k,\xi^k,\nu^k) \Delta w \label{eq:branch_qp_obj}}
	{\label{eq:branch_qp}}{}
	\addConstraint{h(w^k) + \nabla h(w^k)^\top \Delta w}{=0 \label{eq:branch_qp_eq}}
	\addConstraint{g(w^k) +\nabla g(w^k)^\top \Delta w}{\leq0 \label{eq:branch_qp_ineq}}
	\addConstraint{G_i(w^k) + \nabla G_i(w^k)^\top \Delta w = 0,\ H_i(w^k) + \nabla H_i(w^k)^\top \Delta w}{\geq 0,\ \forall i \in \I(\bar{w})\label{eq:branch_qp_comp1}}
	\addConstraint{G_i(w^k) + \nabla G_i(w^k)^\top \Delta w \geq 0,\ H_i(w^k) + \nabla H_i(w^k)^\top \Delta w}{= 0,\ \forall i \in \I^\mathrm{C}(\bar{w}),\label{eq:branch_qp_comp2}}
\end{mini!}
and it converges with the rate~\eqref{eq:sqpcc_contraction_estimate}.

From now on, we omit the subscript $\I$ in the QP solutions, as it will be clear from the context.
Moreover, under the conditions discussed below, the QP solutions $z^{\I,k}$ coincide with the QPCC iterates $z^k$.
If we choose $z^0 \in \mathcal B_{\bar\varepsilon}(\bar z)$ with $\bar\varepsilon := \min_{\I \in \mathcal P(\bar w)} \varepsilon^{\I}$, then the sequence of iterates generated by repeatedly solving the QP~\eqref{eq:branch_qp} with any $\I \in \mathcal P(\bar w)$ (not necessarily the same $\I$ for all $k$) converges to $\bar z$.

Next, we show that solutions of the QP~\eqref{eq:branch_qp} at linearization points $z^{k}$ sufficiently close to $\bar z$ are also S-stationary points of the QPCC~\eqref{eq:qpcc}.
Recall that, for nondegenerate indices, the corresponding inequality constraints are inactive at the solution, that is,
\[
G_i(\bar w) = 0,\quad H_i(\bar w) > 0,\quad \forall i \in \I_{0+}(\bar w) \subseteq \I(\bar w).
\]
By applying Proposition~\ref{th:active_set_stabilization} to ${\BNLP}_{\I}$ and the QP~\eqref{eq:branch_qp}, it follows that, for each constraint pair in~\eqref{eq:branch_qp_comp1}, there exists a constant $\rho_i^H > 0$, $i \in \I_{0+}(\bar w)$, such that, with
$\tilde\rho_{0+} := \min_{i \in \mathcal I_{0+}(\bar w)} \rho_i^H$,
for every $z^k \in \B_{\tilde\rho_{0+}}(\bar z)$ it holds that
\[
G_i(w^k) + \nabla G_i(w^k)^\top \Delta w = 0,\quad
H_i(w^k) + \nabla H_i(w^k)^\top \Delta w > 0,
\quad \forall i \in \I_{0+}(\bar w).
\]
The same argument applies to indices $i \in \I_{+0}(\bar w)$, yielding a constant $\tilde\rho_{+0} := \min_{i \in \I_{+0}(\bar w)} \rho_i^G > 0$ such that, for every $z^{k} \in \B_{\tilde\rho_{+0}}(\bar z)$, it holds that 
\(
G_i(w^k) + \nabla G_i(w^k)^\top \Delta w > 0,\ H_i(w^k) + \nabla H_i(w^k)^\top \Delta w = 0\,  \forall i \in \I_{+0}(\bar w).
\)

Next, suppose that $\bar\xi_i > 0$ and $\bar\nu_i > 0$ for all $i \in \I_{00}(\bar w)$ (this ULSC assumption will be relaxed below).
Then, for each $i \in \I_{00}(\bar w)$, the constraints $H_i(w) \geq 0$ for $i \in \I(\bar w) \setminus \I_{0+}(\bar w)$
and $G_i(w) \geq 0$ for $i \in \I^{\mathrm{C}}(\bar w) \setminus \I_{+0}(\bar w)$ in ${\BNLP}_{\I}$ are strictly active at $\bar w$.
By applying Proposition~\ref{th:active_set_stabilization} to the QP~\eqref{eq:branch_qp} again, there exists a constant $\tilde\rho_{00} > 0$ such that, for all $z^{k} \in \B_{\tilde\rho_{00}}(\bar z)$, it holds that 
\[
G_i(w^k) + \nabla G_i(w^k)^\top \Delta w = 0,\quad H_i(w^k) + \nabla H_i(w^k)^\top \Delta w = 0, \quad \forall i \in \I_{00}(\bar w).
\]
Furthermore, the corresponding multipliers satisfy $\xi_i^{k+1} =  \xi_i(r^k) > 0$ and $\nu_i^{k+1} =  \nu_i(r^k) > 0$.
Therefore, with $\tilde{\rho} = \min(\tilde{\rho}_{00}, \tilde{\rho}_{0+}, \tilde{\rho}_{+0})$ and for all $z^{k} \in \B_{\tilde{\rho}}(\bar z )$ the QP~\eqref{eq:branch_qp} has the same active set at $w^{k+1}$ as $\mathrm{BNLP}_\I$ at $\bar w$.
Moreover, the multipliers associated with the degenerate indices $i \in \I_{00}(\bar w)$ satisfy the conditions of S-stationarity.
By applying Proposition~\ref{prop:uniqe_mpcc_multipliers} to the QP solutions of~\eqref{eq:branch_qp}, we conclude that these points are also solutions of the QPCC~\eqref{eq:mpcc_intro}, for any $\I \in \mathcal P(\bar w)$.

Up to this point, the result has been established for initial points $z^0 \in \B_{\bar \varepsilon}(\bar z) \cap \B_{\tilde\rho}(\bar z)$ under the additional ULSC assumption, for which the complementarity active sets of the QPCC, denoted by
$(\I_{0+}^{\mathrm{qp}}(w^k), \I_{+0}^{\mathrm{qp}}(w^k), \I_{00}^{\mathrm{qp}}(w^k))$, coincide with those of the MPCC,
$(\I_{0+}(\bar w), \I_{+0}(\bar w), \I_{00}(\bar w))$.
In this case, convergence of the SQPCC method follows directly from the convergence of SQP applied to any ${\BNLP}_{\I}$.

\textbf{Part II.} It remains to consider the case where ULSC does not hold and the complementarity active sets of the QPCC do not coincide with those of the MPCC at $\bar w$. 
This situation allows for the choice of a larger constant $\tilde\rho > 0$.
To this end, we discuss next each possible case in which a solution of~\eqref{eq:branch_qp} does not have the same active set as ${\BNLP}_{\I}$ at $\bar w$.

First, we regard, a specific nondegenerate index $j \in \I_{0+}(\bar w)$. 
The SQP method applied to~$\mathrm{BNLP}_\I$ may converge with $z^0 \in \B_{\varepsilon}(\bar z)$ and $\varepsilon > \rho^H_j > \tilde{\rho}_{0+} > 0$.  
In other words, the active set of the nondegenerate pair must not immediately be identified, and we have in~\eqref{eq:branch_qp_comp1}:
\[
G_j(w^k) + \nabla G_j(w^k)^\top \Delta w = 0,\ H_j(w^k) + \nabla H_j(w^k)^\top \Delta w = 0,\  j  \in \I_{0+}(\bar w) \subseteq \I(\bar w).
\]
For the active inequality constraint in~\eqref{eq:branch_qp_comp1}, the associated multiplier already satisfies $\nu_j^{k+1} \geq 0$.
If the multiplier corresponding to the equality constraint satisfies $\xi_j^{k+1} \geq 0$, then the QP solution considered so far still fulfills the conditions of S-stationarity and, by Proposition~\ref{prop:uniqe_mpcc_multipliers}, is also an S-stationary point of the QPCC~\eqref{eq:qpcc}.
We collect all such indices $j \in \I_{0+}(\bar w)$ that do not lead to a violation of S-stationarity in the set $\J_{0+} \subseteq \I_{0+}(\bar w)$.
On the other hand, if $\xi_j^{k+1} < 0$, then the initialization radius must be reduced to exclude these constraints from being active.
Specifically, we choose $z^0 \in \B_{\rho_{0+}}(\bar z)$, where ${\rho}_{0+} = \min_{i \in\I_{0+}(\bar{w}) \setminus \J_{0+}} \rho^H_i$.

By repeating the same argument for the constraints~\eqref{eq:branch_qp_comp2} and indices $i \in \I_{+0}(\bar w) \subseteq \I^{\mathrm C}(\bar w)$, we can analogously refine $\tilde\rho_{+0}$ to 
$
\rho_{+0} := \min_{i \in \I_{+0}(\bar w) \setminus \J_{+0}} \rho_i^G,
$
where $\J_{+0} \subseteq \I_{+0}(\bar w)$ is defined in analogy to $\J_{0+}$.

Similarly, for degenerate indices $i \in \I_{00}(\bar w)$ satisfying $\bar\xi_i > 0$ and $\bar\nu_i > 0$, it may occur, for $\varepsilon > \tilde\rho_{00}$, that for some $\I$ and a specific $j \in \I_{00}(\bar w) \cap \I(\bar w)$ in~\eqref{eq:branch_qp_comp1} we have
\[
G_j(w^k) + \nabla G_j(w^k)^\top \Delta w = 0,\quad
H_j(w^k) + \nabla H_j(w^k)^\top \Delta w > 0 .
\]
In this case, the corresponding complementarity pair is nondegenerate, and all such indices are collected in the set $\J_{00}^{0+} \subseteq \I_{00}(\bar w)$.
Repeating the same argument for the constraints~\eqref{eq:branch_qp_comp2}, we collect the corresponding indices in $\J_{00}^{+0} \subseteq \I_{00}(\bar w)$.
Thus, all indices $j \in \J_{00} := \J_{00}^{0+} \cup \J_{00}^{+0}$ are nondegenerate in the QP solution and do not violate S-stationarity.
Accordingly, we refine $\tilde\rho_{00}$ to 
\(
\rho_{00} := \min_{i \in \I_{00}(\bar w) \setminus \J_{00}} \rho_i .
\)

Altogether, we have established S-stationarity of the QPCC at iterates $w^{k+1} \in \B_{\rho}(\bar w)$,
with $\rho := \min(\rho_{0+}, \rho_{+0}, \rho_{00})$, and with complementarity active sets given by
$\I_{0+}^\mathrm{qp}(w^{k+1}) = \I_{0+}(\bar w) \cup \J_{00}^{0+} \setminus \J_{0+}$,
$\I_{+0}^\mathrm{qp}(w^{k+1}) = \I_{+0}(\bar w) \cup \J_{00}^{+0} \setminus \J_{+0}$ and
$\I_{00}^\mathrm{qp}(w^{k+1}) = \I_{00}(\bar w) \cup \J_{0+}\cup \J_{+0} \setminus \J_{00}$.

It remains to relax the intermediate ULSC assumption.
Consider an index $j \in \I_{00}(\bar w)$ with $\bar\xi_j > 0$ and $\bar\nu_j = 0$, and two BNLPs that differ only in the index $j$, denoted by ${\BNLP}_{\I_a}$ and ${\BNLP}_{\I_b}$.
For $j \in \I_a$, the corresponding linearized constraint in the QP is
\begin{align}\label{eq:branch_qp_case_a}
	G_j(w^k) + \nabla G_j(w^k)^\top \Delta w = 0,\quad	H_j(w^k) + \nabla H_j(w^k)^\top \Delta w \geq 0.
\end{align}
For $j \in \I_b$, it is given by
\begin{align}\label{eq:branch_qp_case_b}
	G_j(w^k) + \nabla G_j(w^k)^\top \Delta w \geq 0,\quad 	H_j(w^k) + \nabla H_j(w^k)^\top \Delta w = 0.
\end{align}
It is sufficient to consider ${\BNLP}_{\I_a}$ and the corresponding QP with the linearized constraint~\eqref{eq:branch_qp_case_a}.
Since $\bar\nu_j = 0$, the inequality constraint is weakly active at the solution $\bar w$ of ${\BNLP}_{\I_a}$.
By Proposition~\ref{th:active_set_stabilization}, this constraint may remain inactive for all finite $k$ and become active only in the limit.
In this case, the index $j$ is nondegenerate in the QPCC, and S-stationarity of the QP solution is preserved.
On the other hand, if the inequality constraint in~\eqref{eq:branch_qp_case_a} is active, then the associated multiplier satisfies $\nu_j^{k+1} \geq 0$.
Moreover, there exists a constant $\rho_j^G > 0$ such that $\xi_j^{k+1} = \xi_j(r^k) \geq 0$ for all $z^{k} \in \B_{\rho_j^G}(\bar z)$.
In this case, the QP solution corresponds to an S-stationary point of the QPCC provided that we update
$\rho_{00} := \min(\rho_{00}^+, \rho_j^G)$, where
$
\rho_{00}^+ := \min_{i \in I_{00}^+(\bar w) \setminus \J_{00}} \rho_i,
$
and
$
\I_{00}^+(\bar w) := \{ i \in \I_{00}(\bar w) \mid \bar\xi_i > 0,\ \bar\nu_i > 0 \}.
$
An analogous argument applies to indices $j \in \I_{00}(\bar w)$ with $\bar\xi_j = 0$ and $\bar\nu_j > 0$, allowing $\rho_{00}$ to be reduced further if necessary.

To complete this part of the proof, we consider complementarity pairs with $j \in \I_{00}(\bar w)$ such that $\bar\xi_j = 0$ and $\bar\nu_j = 0$.
In this case, for both BNLPs indexed by $\I_a$ and $\I_b$, the inequality constraints in \eqref{eq:branch_qp_case_a} and \eqref{eq:branch_qp_case_b} are weakly active at the solution $\bar w$. 
By Proposition~\ref{th:active_set_stabilization}, these constraints may remain inactive for all $k$.
As in the previous cases, the resulting QP solution is an S-stationary point of the QPCC, where the index $j$ is now nondegenerate.
Consider now the case~\eqref{eq:branch_qp_case_a} in which the inequality constraint is active.
Then $\nu_j^{k+1} \geq 0$, and if, in addition, the multiplier associated with the equality constraint satisfies $\xi_j^{k+1} \geq 0$, the QP solution again yields an S-stationary point of the QPCC.
On the other hand, if $\xi_j^{k+1} < 0$, the corresponding complementarity pair is degenerate and does not satisfy the definition of S-stationarity.
Denote the solution of this QP by $z^{\I_a,k+1}$.

Next, we check whether $z^{\I_a,k+1}$ satisfies the KKT conditions of the QP obtained from $\BNLP_{\I_b}$ at $z^{k}$, which contains the constraint pair~\eqref{eq:branch_qp_case_b}.
In this case, the first constraint is an inequality, and it is impossible that the associated multiplier satisfies $\nu_j^{k+1} < 0$.
Hence, this QP yields a solution in which the index $j$ is nondegenerate, and the corresponding point is S-stationary.

Finally, we show that it is impossible for both QPs obtained from $\BNLP_{\I_a}$ and $\BNLP_{\I_b}$ at $z^{k}$ to admit $j$ as a degenerate index with multiplier signs violating the definition of S-stationarity.
Indeed, by Proposition~\ref{prop:uniqe_mpcc_multipliers}, for identical active sets the two QPs must have identical Lagrangians and therefore identical MPCC--Lagrange multipliers.
This contradicts the possibility that $\nu_j^{k+1} < 0$ in one formulation and not in the other.

To summarize, by refining $\tilde\rho$ to $\rho := \min(\rho_{00}, \rho_{0+}, \rho_{+0})$ and choosing the initial point $z^0 \in \mathcal B_{\varepsilon}(\bar z)$ with $\varepsilon := \min(\rho, \bar\varepsilon)$, we have shown that there exists an index set $\I \in \mathcal P(\bar w)$ such that the solution of the QP~\eqref{eq:branch_qp} yields an S-stationary point of the QPCC.
Moreover, any sequence of solutions generated by repeatedly solving such QPs converges to $\bar z$.
Consequently, there exists a sequence of S-stationary points of the QPCC that converges to $\bar z$ with the same local convergence rate as SQP applied to the corresponding $\BNLP_{\I}$.

\textbf{Part III.} 
To conclude the proof, we consider one further case in which the region of local convergence can be enlarged.
By MPCC--SSOSC and MPCC--LICQ, the solution $\bar w$ is an isolated local solution of the MPCC~\cite[Theorem 1]{Hu2002}.
That is, there exists a constant $\sigma > 0$ such that $\bar w$ is the unique local minimizer in $\B_{\sigma}(\bar w)$.

Now consider a point $\tilde w$ that is MPCC--stationary but not B-stationary (and hence not S-stationary), that is, there exists at least one index $i \in \I_{00}(\tilde w)$ such that $\tilde\xi_i < 0$ or $\tilde\nu_i < 0$.
Together with its associated multipliers, we denote the corresponding primal--dual point by $\tilde z$.

If $\tilde w \notin \B_{\sigma}(\bar w)$, we set $U := \B_{\varepsilon}(\bar z)$ and the proof is complete (cf. Fig.~\ref{fig:partIII_regions}(a) for an illustration).
Otherwise, if $\tilde w \in \B_{\sigma}(\bar w)$, we may apply the SQP method to any $\BNLP_{\tilde\I}$ (provided that the SSOSC holds for $\BNLP_{\tilde\I}$ at $\tilde{w}$ and its associated MPCC--Lagrange multipliers; if it does not, the proof is complete and we set $U := \B_{\varepsilon}(\bar w)$) with $\tilde\I \in \tilde{\mathcal P}(\tilde w)$, each of which admits a local convergence radius $\tilde\varepsilon_{\tilde\I}$.
Let $\tilde\varepsilon := \min_{\tilde\I \in \tilde{\mathcal P}(\tilde w)} \tilde\varepsilon_{\tilde\I}$ and suppose that
$z^0 \in \B_{\tilde\varepsilon}(\tilde z)$.
Denote by $r := \|\tilde w - \bar w\|$ the distance between the non-B-stationary point $\tilde w$ and the B-stationary point $\bar w$.

If $r > \varepsilon$, where $\varepsilon$ is the convergence radius obtained in the first two parts of the proof, we again set $U := \B_{\varepsilon}(\bar w)$ and the proof is complete (again illustrated by Fig.~\ref{fig:partIII_regions}(a))
If, on the other hand, $r < \varepsilon$, we examine conditions under which the QP subproblems of $\BNLP_{\tilde\I}$, $\tilde\I \in \tilde{\mathcal P}(\tilde w)$, yield S-stationary points of the QPCC (cf. Fig.~\ref{fig:partIII_regions}(b) and (c)).
Clearly, $\Delta w = 0$ cannot be an S-stationary point of the QPCC at $\tilde w$, since $\tilde w$ itself is not S-stationary.
Hence, it is not possible that $\I^{\mathrm{qp}}_{00}(w^{k+1}) = \I_{00}(\tilde w)$.

Nevertheless, QP solutions may exist that do not contain degenerate indices $j \in \I_{00}(\tilde w)$ with multipliers violating S-stationarity.
By an argument analogous to that used in the first part of the proof and by applying Proposition~\ref{th:active_set_stabilization} to $\BNLP_{\tilde\I}$, there exists a constant $\rho_1 > 0$ such that, for all $z^0 \in \B_{\tilde\varepsilon}(\tilde w) \setminus \B_{\rho_1}(\tilde w)$, the QP solution satisfies
\[
\I_{00}^{\mathrm{qp}}(w^{k+1}) \cap
\{ i \in \I_{00}(\tilde w) \mid \tilde\xi_i < 0 \textnormal{ or } \tilde\nu_i < 0 \} = \emptyset.
\]
That is, the QP solution does not contain degenerate indices whose multipliers invalidate S-stationarity
(red and black dots in Fig.~\ref{fig:partIII_regions}(b) and (c); case of spurious sequence (red dots) is further discussed in Sec.~\ref{sec:leyffer}).

Similarly, indices that are nondegenerate at $\tilde w$ but become degenerate in the QP solution and violate S-stationarity can be excluded by choosing $z^0 \in \B_{\tilde\rho_2}(\tilde z)$ for some $\tilde\rho_2 > 0$, so that the nondegenerate indices of $\tilde w$ stabilize in the QP solution.
Consequently, if we initialize $z^0 \in \tilde U := \B_{\min(\tilde\varepsilon,\rho)}(\tilde z) \setminus \B_{\rho_1}(\tilde z)$, the QP solutions associated with $\BNLP_{\tilde\I}$ are S-stationary points of the QPCC~\eqref{eq:qpcc}.
Moreover, if $r < \varepsilon$, these iterates converge, with the same local convergence rate as SQP, in a finite number of iterations $K$ such that $z^K \in \B_{\varepsilon}(\bar z)$ (denoted by $w^K$ in Fig.~\ref{fig:partIII_regions}).
Combining this argument with the first part of the proof, we obtain convergence of the SQPCC method for any initial point $z^0 \in U := \B_{\varepsilon}(\bar z) \cup \tilde U$.
\qed

Since the convergence rates of SQPCC are inherited from those of the SQP method applied to BNLPs, all conclusions drawn after the proof of Theorem~\ref{th:sqp_local_convergence} for the $\kappa$--$\omega$ framework carry over directly to the SQPCC method.
In particular, the less nonlinear the problem functions are, that is, the smaller $\omega$ is, and the smaller the derivative approximation errors in the QPCC~\eqref{eq:qpcc} are, that is, the smaller $\kappa^k$ is, the larger the local convergence region becomes.
Moreover, smaller values of $\kappa^k$ yield faster local convergence.

\begin{remark}
	Since the proof relies on the local convergence theory of the SQP method, in line with Remark~\ref{rem:sqp_ass}, the assumptions of MPCC--LICQ can be relaxed to MPCC--SMFCQ.  
	Already in the SQP case, further relaxation of the constraint qualifications requires modifications of the QP subproblems~\cite{Wright1998}.
	Therefore, in analogy with SQP convergence results~\cite{Bonnans1994}, it can be expected that our standing assumptions are the weakest assumptions under which a local SQPCC convergence result can be obtained without further regularization of the QPCC subproblems.
\end{remark}

\begin{figure}[t]
	\centering
	\includegraphics[width=\textwidth]{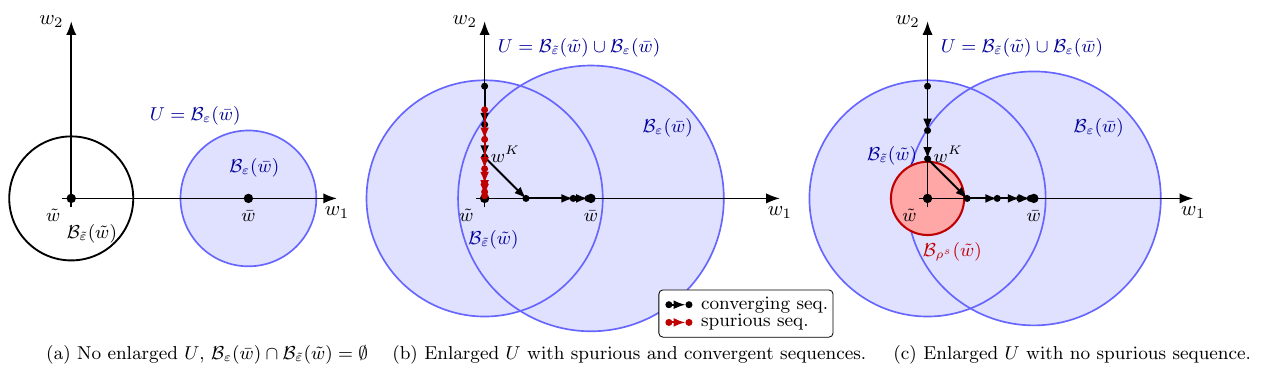}
	\caption{Schematic illustration of the argument in Part~III of the proof of Theorem~\ref{th:sqpcc_local_convergence} in the $(w_1,w_2)$-plane. 
	The reference S-stationary point is denoted by $\bar w$ and the nearby non-S-stationary point by $\tilde w$. 
	In (a), the neighborhoods $\mathcal{B}_{\varepsilon}(\bar w)$ and $\mathcal{B}_{\tilde\varepsilon}(\tilde w)$ do not intersect, so there is no enlargement of $U$. 
	Panels (b) and (c) illustrate the enlarged set $U$ when the two neighborhoods overlap. 
	In (b), red dots indicate spurious sequences attracted to $\tilde w$, while black dots indicate sequences that turn toward $\bar w$. 
	Panel (c) shows the same situation with an inner ball $\mathcal{B}_{\rho^s}(\tilde w)$; cf. Corollary~\ref{cor:no_spurious_sequences}.}

	\label{fig:partIII_regions}
	\vspace{-0.4cm}
\end{figure}

Next we discuss cases where nonconverging sequences of QPCC S-stationary points may exist, and how they can be excluded, and thereby strengthening further the results of Theorem~\ref{th:sqpcc_local_convergence}.

\subsection{On the counterexample of Leyffer and Munson}\label{sec:leyffer}
In~\cite[Section~2.3]{Leyffer2007}, the authors provide an example in which SQPCC can converge to a non-S-stationary solution. 
Here, we revisit this example, show that Theorem~\ref{th:sqpcc_local_convergence} is not contradicted by it, and draw several conclusions from it.

Consider the MPCC~\cite[Eq.~(2.7)]{Leyffer2007}:
\begin{align}\label{eq:leyffer2007}
	\min_{w\in \R^2}\; (w_1-1)^2 + w_2^2 + w_2^3 
	\quad \mathrm{s.t.} \quad
	0 \le w_1 \perp w_2 \ge 0.
\end{align}
This problem satisfies MPCC--LICQ at all feasible points. 
It has a unique minimizer $\bar w = (1,0)$, which is an S-stationary point and satisfies the MPCC--SSOSC. 
Moreover, it also has an M-stationary point $\tilde w = (0,0)$ with MPCC--Lagrange multipliers $\tilde \xi = -2$ and $\tilde \nu = 0$.

Leyffer and Munson observe that there exists a sequence of SQPCC iterates that converges to $\tilde w$ instead of $\bar w$.
 If one initializes at, for example, $w^0 = (t,0)$ with sufficiently small $t>0$, then the SQPCC iterates are
\begin{align*}
w^{k+1} = \left(0, \frac{3 (w_2^k)^2}{6 w_2^k+2}\right),
\end{align*}
and these are S-stationary points of the corresponding QPCCs.
This sequence converges quadratically to the M-stationary point $\tilde w$. 
Locally, these iterates reduce to SQP iterations on the BNLP with $w_1 = 0$ and $w_2 \geq 0$. 
Since $\tilde \nu = 0$, the weakly active constraint $\tilde w_2 = 0$ is identified only in the limit; cf.~Propositions~\ref{th:active_set_stabilization} and \ref{th:active_set_stabilization_sqpcc}.

Applying the exact-Hessian SQP method~\cite{Fletcher2006} directly to the NLP reformulation (cf.~Eq.~\eqref{eq:mpcc_nlp}) of~\eqref{eq:leyffer2007}, using the active-set solver \texttt{qpOASES}~\cite{Ferreau2014} via \texttt{CasADi}~\cite{Andersson2019}, yields the same unfavorable behavior. 
The iterates are illustrated in Fig.~\ref{fig:leyffer2007_convergence} (left). 
Note that in this case the SQP iterates can only follow the BNLP corresponding to the feasible initialization point.

However, this does not contradict Theorem~\ref{th:sqpcc_local_convergence}. 
Under assumptions satisfied by this example, that theorem asserts the existence of at least one sequence of QPCC S-stationary points converging to an S-stationary point of the MPCC. 
The sequence above is not such a sequence. 
We show next that such a sequence still exists and that the conclusion of Theorem~\ref{th:sqpcc_local_convergence} remains valid.

For illustration, we solve the MPCC~\eqref{eq:leyffer2007} with the exact Hessian SQPCC method. 
The QPCC subproblems are solved by a homotopy loop with a Scholtes relaxation~\cite{Scholtes2001}. We use the implementation in \texttt{nosnoc}'s \texttt{mpccsol} function for solving QPCCs, which in turn uses \texttt{Ipopt}~\cite{Waechter2006} to solve the relaxed subproblems.\footnote{The code for the examples is available at \url{https://github.com/nurkanovic/sqpcc_examples}.}
The iterates converge to the MPCC solution $\bar w$, and after identifying the correct branch, the BNLP is a QP, so convergence is completed in one additional iteration; see the middle and right plots in Fig.~\ref{fig:leyffer2007_convergence}.

\begin{figure}[t]
	\centering
	\includegraphics[width=0.95\textwidth]{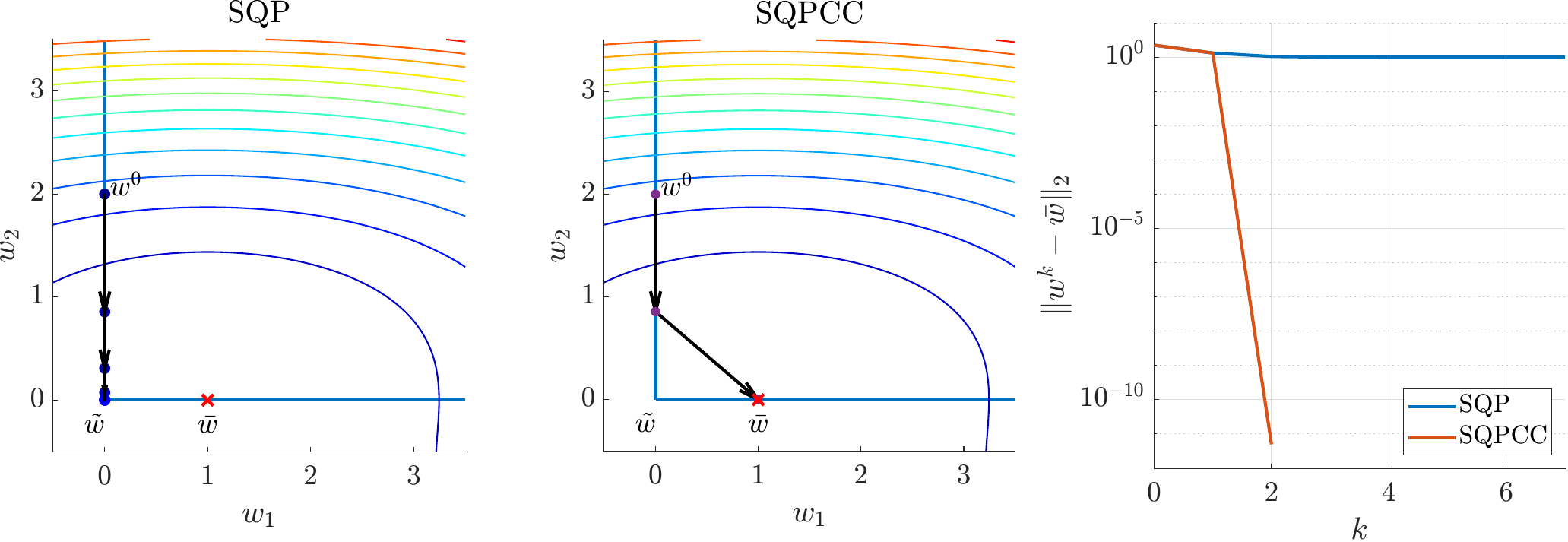}
	\caption{The left plot illustrates the iterates of the exact-Hessian SQP method, and the middle plot those of the exact-Hessian SQPCC method in the $(w_1,w_2)$-plane. 
	The initial point in both cases is $w^0 = (0,2)$.
	The right plot shows the errors over the iterations for two methods; the SQPCC method converges to $\bar w$, and SQP does not.
	}
	\label{fig:leyffer2007_convergence}
	\vspace{-0.5cm}
\end{figure}

The mechanism behind this behavior is explained as follows. 
For $t>0$, the QPCC subproblem at $(0,t)$ reads:
\begin{mini*}[2]
	{\substack{\Delta w \in \R^{2}}}{-2 \Delta w_1 + (2t+3t^2)\Delta w_2 + \Delta w_1^2 + (1+3t)\Delta w_2^2}
	{}{}
	\addConstraint{0 \leq \Delta w_1 \;\perp\; t+\Delta w_2 \geq 0.}
\end{mini*}

Inspection of the two branch QPs shows that there are two S-stationary points:
\begin{equation*}
	\Delta w^{(1)} =
	\left(
	0,\;
	-\frac{2t+3t^2}{2(1+3t)}
	\right),
	\qquad
	\Delta w^{(2)} = (1,-t).
\end{equation*}
The solution $\Delta w^{(1)}$ produces the stalling behavior discussed above, whereas $\Delta w^{(2)}$ yields convergence as predicted by Theorem~\ref{th:sqpcc_local_convergence} and observed in our numerical experiment. 
This highlights an advantage of SQPCC over SQP applied to the NLP reformulation, since SQP cannot recover the step $\Delta w^{(2)}$.

\subsection{On excluding spurious sequences}\label{sec:sqpcc_tight}
The previous example also clarifies when nonconvergent sequences of QPCC S-stationary points may occur. 
The point $\tilde w$ is a local minimizer of the BNLP with $w_1 = 0$ and $w_2 \geq 0$, and Part~III of the proof shows that such points may act as local attractors. 
These SQP iterates for this BNLP are valid SQPCC iterates, since they are S-stationary points of the corresponding QPCCs. 
These iterates, however, identify the complementarity active set at $\tilde w$ only asymptotically because the active constraint is weakly active.

We therefore conclude that if there exists a local minimizer of a BNLP with a weakly active degenerate complementarity constraint that violates the conditions of S-stationarity, then there may exist a sequence of SQPCC iterates that does not converge to an S-stationary point.
Theorem~\ref{th:sqpcc_local_convergence} shows, however, that this does not exclude the simultaneous existence of another locally convergent SQPCC sequence.
This behavior can be excluded by imposing an additional assumption on $\tilde w$.

\begin{corollary}\label{cor:no_spurious_sequences}
	Suppose that the assumptions of Theorem~\ref{th:sqpcc_local_convergence} hold at the S-stationary point $\bar z = (\bar w,\bar\lambda,\bar\mu,\bar\xi,\bar\nu)$ of the MPCC~\eqref{eq:mpcc_intro}.
	Assume, in addition, that every non-S-stationary point
	$\tilde z = (\tilde w,\tilde\lambda,\tilde\mu,\tilde\xi,\tilde\nu)$
	sufficiently close to $\bar w$ satisfies
	\begin{align}\label{eq:no_spurios_condition}
		\tilde \xi_i \neq 0
		\quad\text{and}\quad
		\tilde \nu_i \neq 0
		\qquad \forall i \in \I_{00}(\tilde w).
	\end{align}
	Then there exists no sequence of SQPCC iterates starting in $U$ that can converge to a non-S-stationary point $\tilde w$.
	If, in addition, all QPCC iterates are chosen as S-stationary points contained in the neighborhood $U$ from Theorem~\ref{th:sqpcc_local_convergence}, then every such sequence converges to $\bar z$.
\end{corollary}

\textit{Proof.}
By Theorem~\ref{th:sqpcc_local_convergence}, there exists a neighborhood $U$ of $\bar z$ such that, for every initial point $z^0 \in U$, there exists a sequence of QPCC S-stationary points in $U$ converging to $\bar z$.

It remains to exclude convergence to non-S-stationary points.
By Part~III of the proof of Theorem~\ref{th:sqpcc_local_convergence}, a sequence of SQPCC iterates can converge to a non-S-stationary point $\tilde w$, only if it is a local minimizer of an associated branch NLP and the corresponding SQP iterates yield QPCC S-stationary points for each $k$.
This is only possible if $\I^{\mathrm{qp}}_{00}(w^k) \cap \I_{00}(\tilde w) = \emptyset$ for all $k$, since otherwise the conditions of S-stationarity for the QPCC would be violated.
Biactive complementarity constraints are weakly active constraints in the BNLP, and if \eqref{eq:no_spurios_condition} holds, then by Proposition~\ref{th:active_set_stabilization}, there exists a $\rho^{{s}}>0$ and a finite $K$ such that, for all $w^K \in \mathcal{B}_{\rho^{{s}}}(\tilde w)$, it follows that $\I^{\mathrm{qp}}_{00}(w^{K+1}) = \I_{00}(\tilde w)$.
Further, by Corollary~\ref{th:fiacco_nlp_sensitivity}, $w^{K+1}$ must violate the conditions of S-stationarity; compare Fig.~\ref{fig:partIII_regions}(b) and (c) for an illustration of the argument.
Therefore, there cannot exist a sequence of iterates converging to $\tilde w$.

For the second claim, let $\{z^k\} \subset U$ be any sequence of QPCC iterates chosen as S-stationary points.
By the first part, such a sequence cannot converge to a non-S-stationary point.
The proof of Theorem~\ref{th:sqpcc_local_convergence} explicitly constructs all QPCC S-stationary points in $U$, and since for any $z^k \in U$ it also follows that $z^{k+1} \in U$, all such sequences converge to $\bar z$.
\qed

The M-stationary point $\tilde w$ in the example~\eqref{eq:leyffer2007} violates the conditions of Corollary~\ref{cor:no_spurious_sequences} and a spurious sequence exists.
On the other hand, in Example~\ref{ex:sqpcc_local_convergence} the C-stationary  point $\tilde w$ has no zero multiplier, and no spurious sequence exists.
In addition, such cases can be excluded entirely from the local convergence analysis by restricting attention to the smaller local convergence region determined in Parts~I and II of the proof of Theorem~\ref{th:sqpcc_local_convergence}.
Note that this smaller region still allows for complementarity active-set changes and requires no ULSC.
In practice, the assumptions of Corollary~\ref{cor:no_spurious_sequences} can be satisfied by warm-starting a local QPCC solver so that the selected S-stationary points remain in $U$, since for every $z^k \in U$ there exists a corresponding $z^{k+1} \in U$.

Finally, we note that under ULSC, by applying MPCC perturbation theory to the QPCC~\cite[Theorem~11]{Scheel2000}, in analogy with applying Theorem~\ref{th:fiacco_nlp_sensitivity} to Eq.~\eqref{eq:perturbed_nlp}, one can conclude that the QPCC solutions close to an S-stationary point are locally unique.
This case was already covered in~\cite{Scholtes2004}, but in contrast to the more general Theorem~\ref{th:sqpcc_local_convergence}, it does not allow for complementarity active-set changes.

Without ULSC, local uniqueness cannot be expected, but, as covered by Theorem~\ref{th:sqpcc_local_convergence} and Corollary~\ref{cor:no_spurious_sequences}, this does not cause difficulties.
An S-stationary point of an MPCC is a KKT point of each associated branch NLP, and, as the proofs show, in the SQPCC method it can be approached along any of these branches, which leads to nonunique convergent sequences.
This is illustrated in the following example.

\begin{example}
	Consider the MPCC
	\begin{align}\label{eq:mpcc_active_set}
		\min_{w\in \R^2}\; w_1^4 + w_1^2 + w_2^4 + w_2^2
		\quad \mathrm{s.t.} \quad
		0 \le w_1 \perp w_2 \ge 0.
	\end{align}
	This problem has a unique S-stationary point $\bar w = (0,0)$ with multipliers $\bar \xi = 0$ and $\bar \nu = 0$.
	Initializing at $w^0 = (t,0)$ or $w^0 = (0,t)$ for sufficiently small $t$ yields, on the two branches, the QPCC S-stationary point sequences
	\[
	w_a^{k+1} = \left(\frac{4(w_1^k)^3}{6(w_1^k)^2+1},\,0\right),
	\qquad
	w_b^{k+1} = \left(0,\,\frac{4(w_2^k)^3}{6(w_2^k)^2+1}\right).
	\]
	Both sequences converge to $\bar w$.
	Furthermore, at each iteration the QPCC subproblem admits S-stationary points on both branches, and any sequence obtained by selecting one of these branch-generated iterates at each step is a valid SQPCC iteration sequence and converges to $\bar w$.
	Note that the optimal active set is identified only in the limit; this is discussed further in the next section.
\end{example}

\subsection{Active-set stabilization for the SQPCC method}\label{sec:sqpcc_active_set}

At the solution $\bar w$, the active sets of the QPCC and the MPCC agree.
However, the proof of the local convergence theorem reveals that SQPCC iterates $z^{k} \in \B_{\varepsilon}(\bar z) \subseteq U$ satisfy:
\begin{align*}
	&\I_{0+}^{\mathrm{qp}}(w^k) \subseteq \I_{0+}(\bar{w}) \cup \I_{00}(\bar{w}),\\
	&\I_{+0}^{\mathrm{qp}}(w^k) \subseteq \I_{+0}(\bar{w}) \cup \I_{00}(\bar{w}),\\
	&\I_{00}^{\mathrm{qp}}(w^k) \subseteq \I_{00}(\bar{w}) \cup \I_{+0}(\bar{w}) \cup \I_{0+}(\bar{w}).
\end{align*}
It is natural to expect that the QPCC active sets stabilize in some sense.
We derive a corresponding active-set stabilization result for SQPCC, analogous to Proposition~\ref{th:active_set_stabilization}.
The proof again exploits the piecewise structure of the MPCC and essentially applies Proposition~\ref{th:active_set_stabilization} to each $\BNLP_\I$.

Using the PULSC condition (cf. Def.\ref{def:pulsc}), we split the set of degenerate indices into
$\I_{00}^0(\bar w, \bar \xi, \bar \nu) = \{i \in \I_{00}(\bar w) \mid \bar \xi_i = 0 \textnormal{ and } \bar \nu_i  = 0 \}$, and 
$\I_{00}^+(\bar w, \bar \xi, \bar \nu) = \{i \in \I_{00}(\bar w) \mid \bar \xi_i > 0 \textnormal{ or } \bar \nu_i  > 0 \}$.

\begin{proposition}[SQPCC active-set stabilization]\label{th:active_set_stabilization_sqpcc}
	Suppose that the assumptions of Theorem~\ref{th:sqpcc_local_convergence} hold.
	Denote by $z^{k}$ the iterates of the SQPCC method, where $w^k$ is an S-stationary point of the QPCC~\eqref{eq:qpcc}, and 
	$(\lambda^k,\mu^k, \xi^k, \nu^k)$ are the associated multipliers.
	Let $z^0 = (w^0,\lambda^0,\mu^0, \xi^0, \nu^0 )\in U$.
	Then there exist constants $\rho^{\mathrm{s}}_1>0$ and $\rho^{\mathrm{s}}_2>0$ and a sequence of SQPCC iterates such that, after finitely many iterations: 
	\begin{enumerate}[(i)]
		\item the iterates satisfy
		$
		z^{k} \in \mathcal{B}_{\rho^{\mathrm{s}}_1}(\bar z) \cap U,
		$
		and the inequality constraints active set stabilizes in the sense that
		\begin{align}\label{eq:sqpcc_active_set_stabilization_std}
			\A_+(\bar w,\bar\mu)= \A_+(w^k,\mu^k),
			\qquad
			\A_+(\bar w,\bar{\mu}) \subseteq  \A(w^k)\subseteq
			\mathcal{A}_+(\bar w,\bar\mu)\cup\mathcal{A}_0(\bar w,\bar\mu).
		\end{align}
		
		If, in addition, strict complementarity holds at $\bar w$, then
		$
		\mathcal{A}(w^k)=\mathcal{A}(\bar w).
		$
		\item The iterates satisfy
		$
		z^{k} \in \mathcal{B}_{\rho^{\mathrm{s}}_2}(\bar z) \cap U,
		$
		and the complementarity constraints active set stabilizes in the sense that:
		\begin{subequations}\label{eq:sqpcc_active_set_stabilization_comp}
			\begin{align}
				&\I_{0+}(\bar w)\subseteq  \I_{0+}^{\mathrm{qp}}(w^k) \subseteq \I_{0+}(\bar w) \cup \I_{00}^0(\bar w, \bar \xi, \bar \nu), \label{eq:sqpcc_active_set_stabilization_comp_nondeg_0}\\
				&\I_{+0}(\bar w)\subseteq  \I_{+0}^{\mathrm{qp}}(w^k) \subseteq \I_{+0}(\bar w) \cup \I_{00}^0(\bar w, \bar \xi, \bar \nu), \label{eq:sqpcc_active_set_stabilization_comp_nondeg_1}\\
				&\I_{00}^+(\bar w, \bar \xi, \bar \nu) \subseteq \I_{00}^{\mathrm{qp}}(w^k) \subseteq \I_{00}^+(\bar w, \bar \xi, \bar \nu) \cup  \I_{00}^0(\bar w, \bar \xi, \bar \nu) \label{eq:sqpcc_active_set_stabilization_comp_deg}.
			\end{align}
		\end{subequations}
		If in addition the PULSC condition holds at $\bar z$, then $\I_{00}(w^k) = \I_{00}^{\mathrm{qp}}(w^k)$. 
	\end{enumerate}
	If $z^{k} \in \mathcal{B}_{\rho^{\mathrm{s}}_2}(\bar z) \cap \mathcal{B}_{\rho^{\mathrm{s}}_1}(\bar z) \cap U$, then both estimates~\eqref{eq:sqpcc_active_set_stabilization_std} and \eqref{eq:sqpcc_active_set_stabilization_comp} hold true.
\end{proposition}
\textit{Proof.}
Let $z^{k} \in \mathcal B_{\varepsilon}(\bar z) \subseteq U$.
Then, by Theorem~\ref{th:sqpcc_local_convergence}, there exist an S-stationary point of the QPCC that is also a solution of the QP~\eqref{eq:branch_qp}.
Consequently, we may apply Proposition~\ref{th:active_set_stabilization} to any $\BNLP_\I$, which immediately yields claim~(i).

For part~(ii), we observe that, in any $\BNLP_\I$, the constraints $G_i(w) = 0,\ H_i(w) \geq 0$, $i \in \I \cap \I_{0+}(\bar w)$,
satisfy that the inequality constraint is inactive at the solution, that is, $H_i(\bar w) > 0$.
An analogous argument applies to the constraints $G_i(w) \geq 0,\ H_i(w) = 0$, $i \in \I^{\mathrm C} \cap \I_{+0}(\bar w)$.
It then follows from Proposition~\ref{th:active_set_stabilization} that there exists $\rho_2^{\mathrm s} > 0$ such that the inactive constraints are identified in a finite number of iterations, which yields the lower bounds in \eqref{eq:sqpcc_active_set_stabilization_comp_nondeg_0}~and~\eqref{eq:sqpcc_active_set_stabilization_comp_nondeg_1}.

Next, we observe in the constraints $G_i(w) = 0,\ H_i(w) \geq 0$, $i \in \I \cap \I_{00}(\bar w)$, or
$G_i(w) \geq 0,\ H_i(w) = 0$, $i \in \I^{\mathrm C} \cap \I_{00}(\bar w)$ in any $\BNLP_\I$, the inequality constraints are active at $\bar w$.
Moreover, for indices $i \in \I_{00}^+(\bar w,\bar\xi,\bar\nu)$, these constraints are strictly active and, by Proposition~\ref{th:active_set_stabilization}, are identified in finitely many iterations, which yields the lower bound~\eqref{eq:sqpcc_active_set_stabilization_comp_deg}.
The indices $i \in \I_{00}^0(\bar w,\bar\xi,\bar\nu)$ correspond to weakly active constraints in a $\BNLP_\I$ and may therefore be identified only asymptotically, which gives the upper bound in~\eqref{eq:sqpcc_active_set_stabilization_comp}.

Under PULSC, for each $i \in \I_{00}(\bar w)$ at least one of $\bar\xi_i$ and $\bar\nu_i$ is strictly positive.
Assuming without loss of generality that $\bar\nu_i > 0$ for all such indices and choosing $\I \in \mathcal P(\bar w)$ with $\I_{00}(\bar w) \subseteq \I$, the inequality constraint in  $G_i(w)=0$, $H_i(w)\geq0$ are strictly active at the solution $\bar w$, and hence finitely identified.
\qed

There is no a priori ordering between the stabilization radii $\rho_1^{\mathrm{s}}$ and $\rho_2^{\mathrm{s}}$.
The size of $\rho_1^{\mathrm{s}}$ is determined by the values of $g_i(\bar w)$ and $\bar\mu_i$ for the inequality constraints, while the size of $\rho_2^{\mathrm{s}}$ is determined by the values of $G_i(\bar w)$ and $H_i(\bar w)$ for $i \in \I_{0+}(\bar w) \cup \I_{+0}(\bar w)$, as well as by $\bar\xi_i$ and $\bar\nu_i$ for $i \in \I_{00}(\bar w)$.
If strict complementarity and PULSC hold, then there exists a converging SQPCC sequence for which all active sets are identified correctly after a finite number of iterations and, depending on the corresponding radii, possibly already at the initialization point.
Note that the nondegenerate indices $i \in \I_{0+}(\bar w) \cup \I_{+0}(\bar w)$ are always identified in finitely many iterations and yield the same inactive constraints in each $\BNLP_\I$. 
Only the degenerate indices may vary, and depending on the value of their optimal  MPCC--Lagrange multipliers, they may be identified either finitely or only asymptotically.

\section{Conclusions and future work}\label{sec:conclusion}
Sequential quadratic programming with complementarity constraints (SQPCC) is a natural extension of the classical sequential quadratic programming (SQP) method to mathematical programs with complementarity constraints (MPCCs), since it retains the complementarity structure directly in the quadratic subproblems.
This paper has presented a new local convergence result for SQPCC, strengthening earlier results while weakening their assumptions.
Our main result establishes local convergence of the basic SQPCC method to S-stationary points, under the assumptions of MPCC--SSOSC and MPCC--LICQ, where the latter can be relaxed directly to MPCC--SMFCQ.
In addition, we work under a Lipschitz assumption on the subproblem functions (the $\omega$-condition) and a compatibility assumption on the approximation error in the subproblem data (the $\kappa$-condition).
These assumptions provide an intuitive way to understand how the nonlinearity of the problem functions and the size of the approximation error influence both the local convergence region and the resulting convergence speed.

Our analysis builds on the local convergence theory of the SQP method.
As part of this development, compared to~\cite{TranDinh2012b}, we also established a local convergence result for SQP under weaker assumptions in the $\kappa$--$\omega$ setting.
In contrast to earlier results, the SQPCC theory developed here does not require strict complementarity-type assumptions, and allows standard and complementarity active-set changes along the iterates.
Moreover, the active-set stabilization results for SQP and SQPCC identify conditions on the Lagrange multipliers that characterize whether the optimal active set is identified in finitely many iterations or only asymptotically, both for standard inequality constraints and for complementarity constraints.

A distinctive feature of the SQPCC method is that non-S-stationary points may enlarge the local convergence region.
At the same time, this phenomenon can also lead to sequences converging to spurious solutions.
By revisiting the counterexample of Leyffer and Munson~\cite{Leyffer2007}, we showed that this does not contradict our local convergence result, and we identified conditions on the MPCC--Lagrange multipliers under which such spurious attracting sequences cannot occur.
In this way, the analysis clarifies both the additional flexibility and the additional subtleties that arise when complementarity structure is retained directly in the subproblems.

We also emphasize the relevance of SQPCC for dynamic optimization applications, in particular for model predictive control, as discussed in detail in ~\cite{Nurkanovic2026a}.
This application area is a major motivation for a better understanding of the convergence properties of SQPCC.
Unlike SQP methods applied to NLP reformulations of MPCCs, SQPCC iterates retain a direct relation to the computation of directional derivatives of the solution map of parametric MPCCs.
Furthermore, the possibility of discovering new complementarity active sets, as covered by our theory, plays an important role in dynamic optimization~\cite{Nurkanovic2026a}.

A natural direction for future work is to weaken the constraint qualification assumptions further, for example to MPCC-MFCQ.
As in the classical SQP setting~\cite{Wright1998}, we expect that such a development will require regularization of the QPCC subproblems.
Without MPCC-MFCQ, not all B-stationary points are S-stationary, and one may expect that convergence to B-stationary points of the MPCC will require the computation of B-stationary points of the QPCC subproblems.
This, in turn, may require crossover and active-set strategies within QPCC solvers of the type developed in~\cite{Kirches2022,Leyffer2007,Nurkanovic2025}.

\begin{acknowledgements}
This research was supported by DFG via projects 504452366 (SPP 2364), 560056112 (robust MPC), 535860958 (ALeSCo) and 525018088 (MAWERO), and by BMWK via 03EN3054B.
The author thanks M. Diehl for stimulating discussions on the local convergence of SQP methods, and A. E. Pozharskiy for the implementation of \texttt{mpccsol} in \texttt{nosnoc}.
\end{acknowledgements}


\end{document}